\documentclass[a4paper,10pt,oneside]{article}

\usepackage{alltt}
\usepackage{subfigure}
\usepackage{amsmath,amsthm,amsfonts,amssymb}
\usepackage{amscd}
\usepackage{graphics}
\usepackage{graphicx,epsfig}
\usepackage{psfrag}
\usepackage{makeidx}
\usepackage[bottom]{footmisc}
\usepackage[margin=1.2cm]{geometry}
\usepackage{parskip}

\theoremstyle{plain}
\newtheorem{thm}{Theorem}
\newtheorem{lem}{Lemma}
\newtheorem{prop}{Proposition}
\newtheorem{cor}{Corollary}
\theoremstyle{definition}
\newtheorem{defn}{Definition}

\newtheorem{exmp}{Example}
\newtheorem{assum}{Assumption}
\newtheorem{rem}{Remark}
\newtheorem{prob}{Problem}
\newtheorem{alg}{Algorithm}

\usepackage{tikz}
\usetikzlibrary{decorations.pathreplacing,calc}
\usetikzlibrary{matrix,decorations.pathreplacing,calc}

\newcommand{\fin}{\hfill$\Box$}
\newcommand{\n}{\normalfont}

\DeclareMathOperator{\diag}{diag}

\title{Families of moment matching based, structure preserving approximations for linear port Hamiltonian systems \thanks{This work is supported by the EPSRC grant "Control For Energy and Sustainability", grant reference EP/G066477/1.}}

\author{T. C. Ionescu\thanks{T. C. Ionescu is with the Department of Electrical and Electronic Engineering, Imperial College London, SW7 2AZ, London, UK. E-mail: {\tt t.ionescu@imperial.ac.uk}.}, A. Astolfi\thanks{A. Astolfi is with the Department of Electrical and Electronic Engineering, Imperial College London, SW7 2AZ, London, UK and with Dipartimento di Ingegneria Civile e Ingegneria Informatica, Universit\`{a} di Roma Tor Vergata, Roma, 00133, Italy., Universit\`{a} di Roma Tor Vergata, 00133 Roma, Italy. E-mail: {\tt a.astolfi@imperial.ac.uk}.}}

\begin{document}

\maketitle                        
                                         
\begin{abstract}
In this paper we propose a solution to the problem of moment matching with preservation of the port Hamiltonian structure, in the framework of time-domain moment matching. We characterize several families of parameterized port Hamiltonian models that match the moments of a given port Hamiltonian system, at a set of finite interpolation points. We also discuss the problem of Markov parameters matching for linear systems as a moment matching problem for descriptor representations associated to the given system, at zero interpolation points. Solving this problem yields families of parameterized reduced order models that achieve Markov parameter matching. Finally, we apply these results to the port Hamiltonian case, resulting in families of parameterized reduced order port Hamiltonian approximations.
\end{abstract}

{\bf Keywords:}                           
Model approximation, Model reduction, Physical models, Markov parameters, System order reduction.  

\section{Introduction}\label{intro}

Port Hamiltonian systems represent an important class of systems used in modelling, analysis and control of physical systems, see e.g. \cite{vdschaft-2000,ortega-vdschaft-maschke-escobar-2002}. These representations are used in lumped parameter system analysis and control stemming from e.g., mechanical systems, electrical systems, electromechanical systems, power systems. One of the main features of such systems is the passivity property, i.e., the internal energy of the system at some time is lower than or equal to the sum of the stored energy in the past and the externally supplied energy in the interval between the past and the present (see, e.g., \cite{vdschaft-2000}). Passivity is essential for stability analysis and controller synthesis represented by techniques such as Passivity Based Control (PBC) and Interconnection and Damping Assignment (IDA), see e.g., \cite{ortega-garciacanseco-2004}. For instance, in the case of large power systems, consisting of synchronous machines interconnected through transmission lines, the passivity (dissipativity) analysis of its port Hamiltonian modelling is important for the stability assessment of such systems and for further control design, see e.g., \cite{kundur-1994,giusto_ortega_stankovic_2006}.
However, physical modelling often leads to (port Hamiltonian) systems of high dimension, usually difficult to analyse and simulate and unsuitable for control design. Hence, model reduction is called for. In the systems and control literature there are many results on model order reduction with preservation of properties and/or structure. For example, passivity preserving model reduction is discussed in e.g., \cite{antoulas-SCL2005,sorensen-SCL2005,ionescu-fujimoto-scherpen-SCL2008,feldman-freund-1995,feldman-freund-1995} and structure preservation is considered in e.g., \cite{fernando-nicholson-1983,ionescu-fujimoto-scherpen-TAC2011} for symmetric systems, \cite{lall-krysl-marsden-PDNP2003} for mechanical systems, \cite{fujimoto-2006,hartmann-MMOD2009,polyuga-vdschaft-MTNS2008} for port Hamiltonian systems.

In the problem of model reduction moment matching techniques represent an efficient tool, see e.g. \cite{grimme-1997,byrnes-lindquist-SIAM2008,antoulas-2005,vandooren-1995,feldman-freund-1995,jaimoukha-kasenally-SIAM1997,antoulas-sorensen-1999} for a complete overview for linear systems. Using a numerical approach based on Krylov projection methods the (reduced order) model is obtained by efficiently constructing a lower degree rational function that approximates a given transfer function (assumed rational). The low degree rational function matches the given transfer function at various points in the complex plane. If the interpolation points are at zero, then the Pad\'e approximation problem is solved. If the interpolation points are finite, then the general rational interpolation problem is solved. In the case of multiple-input multiple output (MIMO) systems the problem is called tangential interpolation, i.e., finding an approximation that interpolates a transfer matrix at selected points along selected directions, see \cite{gallivan-vandendorpe-vandooren-SIAM2004,antoulas-ball-kang-willems-LINALG1990} for further details. If the interpolation points are at infinity, the problem is called partial realization and has been studied in e.g., \cite{mayo-antoulas-LAA2007} and references therein.

Alternatively, for single-input single-output (SISO) systems, a system theoretic, time-domain approach to moment matching has been taken in \cite{astolfi-TAC2010}. In short, the notion of moment of a linear, minimal system has been related to the unique solution of a Sylvester equation, see also, e.g., \cite{gallivan-vandendorpe-vandooren-MTNS2006,gallivan-vandendorpe-vandooren-JCAM2004}, for previous results. Furthermore, the moments are in one-to-one relation with the steady-state response (provided it exists) of the given system driven by a signal generator, which "contains" the interpolation points. The moments have also been connected to the solution of a dual Sylvester equation, and shown to be in one-to-one relation with the well-defined steady-state response of the given system driving a (generalized) signal generator. Used for model reduction, the time-domain approach yields simple and direct characterizations of all parameterized, reduced order models that match a prescribed set of moments of a given system at a set of {\it finite} interpolation points. The classes of reduced order models that achieve moment matching contain subclasses of models that meet additional constraints, i.e., the free parameters are useful for enforcing properties such as, e.g., passivity, stability or relative degree, irrespective of the choice of interpolation points.

Recently, in \cite{wolf-lohmann-eid-kotyczka-EJC2010} the rational interpolation problem for linear port Hamiltonian systems has been addressed using Krylov projection methods, yielding reduced order models that match the moments of the given port Hamiltonian system at a set of prescribed finite or infinite interpolation points. Improved procedures for MIMO systems have been developed in \cite{polyuga-vdschaft-AUT2010,gugercin-polyuga-beattie-vdschaft-ARXIV2011}, where a near-optimal port Hamiltonian approximation that satisfies a set of tangential interpolation conditions is proposed. Furthermore, in \cite{polyuga-vdschaft-ECC2009,polyuga-vdschaft-AUT2010} the partial realization problem for port Hamiltonian systems has been considered. The procedure therein involves finding a change of coordinates such that the Hamiltonian becomes the square of the norm of the state vector, since a direct application of the Krylov methods does not yield a port Hamiltonian approximant. Then the Krylov projections are computed and applied to the system in the new coordinates, resulting in a port Hamiltonian model that matches the Markov parameters of the given system. Note that the Hamiltonian of the approximant is again the square of the norm of the reduced order state vector.

In this paper we study the problem of computing low order approximations that match a set of prescribed moments at a set of finite or infinite points, of a given SISO, port Hamiltonian, linear system and preserve the port Hamiltonian structure, in the framework of time-domain moment matching, as described in the following problem. 
\\
\begin{prob}[General Approximation Problem]\label{prob_redmod_prose}\n
Given a linear, port Hamiltonian, SISO system,
\begin{enumerate}
\item compute {\it the families} of {\it parameterized} reduced order models that match the moments of a given port Hamiltonian system at a set of finite or infinite interpolation points (in this case the models match a set of prescribed Markov parameters).

\item characterize the families of reduced order models which satisfy the following properties: the models match the moments of the given port Hamiltonian system and {\it preserve the port Hamiltonian structure}. In other words, from the classes of models that achieve moment matching we find the reduced order models that inherit the port Hamiltonian form.\fin
\end{enumerate}
\end{prob}

In the case of matching at finite interpolation points we obtain families of parameterized {\it state-space}, reduced order port Hamiltonian models that approximate the given port Hamiltonian system. All the reduced order state-space models share the same transfer function. In the SISO case, the state space parameters are used to enforce additional structure constraints such as diagonal Hamiltonian function, diagonal dissipation, etc. However, at the moment we are not able to determine which of the original variables and their meaning is retained in the reduced order model. In the MIMO case the free parameters can be used to define appropriate directions such that the reduced order models satisfy prescribed sets of tangential interpolation conditions. Regarding the computational aspect, we establish connections between the models from the families of {\it parameterized} port Hamiltonian reduced order models that achieve moment matching and the counterpart approximations obtained using Krylov projections, i.e., there is a relation between the free parameter and the Krylov projector.

Since, to the best of the authors' knowledge, the existing moment matching techniques do not yield an a priori approximation error bound and assuming the interpolation points are free parameters, we can use existing results from the literature to choose these parameters in order to improve the accuracy of the approximation, e.g., in the sense that the $H_2$ norm of error decreases, although it is not the scope of the paper.

We study the problem of Markov parameters matching, i.e., the partial realization problem, by extending the time-domain moment matching results to the case of interpolation points at infinity. We define the notion of moment for a class of linear, descriptor representations, associated to the transfer function of the given system, in terms of the (unique) solutions of {\it generalized} Sylvester equations and their dual counterparts. In particular, the Markov parameters of a given system are the moments of the associated descriptor realization at zero. Furthermore, we relate the moments to the steady-state response, provided it exists, of the descriptor realization driven by/driving signal generators. Performing model reduction we obtain several families of parameterized, (descriptor) reduced order models that match a set of prescribed moments of the descriptor realization associated to the transfer function of a given linear system. In particular, matching at zero yields classes of reduced order models that match the Markov parameters of the given linear system. As in the rational interpolation case, we establish relations between the {\it parameter} which define the family of the port Hamiltonian approximation that matches a set of Markov parameters and a Krylov projector, giving insight into the computational issue of the proposed solution.
\\
Finally, we apply these results to linear port Hamiltonian systems, resulting in {\it families} parameterized of state-space, reduced order port Hamiltonian models that match the Markov parameters of the given port Hamiltonian system. Note that the free parameters can be used to enforce additional physical structure to the approximant. We mention that, to our knowledge, there is no structured procedure to determine the number of interpolation points needed. However, based on \cite{ionescu-scherpen-iftime-astolfi-MTNS2012} the number of interpolation points can be related to the order of the dynamics associated to the higher order Hankel singular values of the given system. To conclude, we mention that the scope of the paper is not to address computational issues, but to propose a system-theoretical based framework for port Hamiltonian reduced order modelling, consistent with the existing theory and suitable for future \textit{nonlinear} extension.

To give the complete picture, we compute several (equivalent) families of reduced order port Hamiltonian models, starting from the unique solution of a Sylvester equation and its {\it dual counterpart}. The families exhibit the same type of properties, without additional advantages or disadvantages, obtained however through different parameterizations. This argument is consistent with the rational Krylov projection modelling where the left and the right projection are dual to each other.

The paper is organized as follows. In Section \ref{sect_prel} we give a brief overview of the definition of moments and moment matching for linear port Hamiltonian systems, as well as of the family of parameterized reduced order models that achieve moment matching at a set of finite interpolation points. We also recall the procedures to obtain a port Hamiltonian approximation using Krylov projections for both SISO and MIMO systems. We present existing choices of interpolation points which yield accurate (in some sense) approximations. The section is completed with the formulation of the model reduction problem to be solved, a particular case of the \ref{prob_redmod_prose} for matching at finite interpolation points. In Section \ref{sect_MarkovM} we study the general problem of time-domain moment matching for a class of descriptor representations associated to the transfer function of a given linear system. In particular we show that matching the moments of the descriptor realization at zero is equivalent to matching the Markov parameters of the given system. Moreover, we obtain the classes of reduced order models that match a set of the Markov parameters of the given system. These models are then related to their Krylov projection counterparts. Based on Section \ref{sect_prel}, in Section \ref{sect_pH_match} we discuss the problem of moment matching at a set of finite interpolation points, with preservation of the port Hamiltonian structure, and characterize the port Hamiltonian reduced order models. Furthermore, we give a necessary and sufficient condition for a reduced order model that achieves moment matching to be a port Hamiltonian model. We describe the families of parameterized state-space port Hamiltonian approximations, the free parameters of which can be used to enforce further properties. We also present a procedure that allows the computation of such a family of models. Based on the results from Section \ref{sect_MarkovM}, in Section \ref{sect_pH_MarkovM} we obtain the classes of reduced order, port Hamiltonian models that match the Markov parameters of a given port Hamiltonian system, proposing a procedure that allows the computation of such families of models. In Section \ref{sect_example}, we give an example that illustrates the results. The paper is completed by a Conclusions section.

Preliminary results have been presented in \cite{ionescu-astolfi-CDC2011}. This paper constitutes a complete version of the aforementioned conference paper, with results presented, discussed and proven in detail. Furthermore, in this paper we develop new results based on the solution of the dual Sylvester equation resulting in an extended system-theoretic characterization of the moment matching for port Hamiltonian systems. In addition, a discussion of the multiple-input multiple-output case is given. Moreover, connections with previous work and results from the literature, especially Krylov projection-based modelling, are established and examples are given to illustrate the theory.

Finally, note that this paper is a preliminary step to develop a time-domain moment matching based model reduction theory for nonlinear port Hamiltonian systems, see e.g., \cite{ionescu-astolfi-NOLCOS2013} for the first results in that direction.

\paragraph*{Notation.} $\mathbb{R}$ is the set of real numbers and $\mathbb{C}$ is the set of complex numbers. $\mathbb{C}^0$ is the set of complex numbers with zero real part and $\mathbb{C}^-$ denotes the set of complex numbers with negative real part. $A^*\in\mathbb{C}^{n\times m}$ denotes the transpose and complex conjugate of the matrix $A\in\mathbb{C}^{m\times n}$. If $A$ is a real matrix, then $A^*$ is the transpose of $A$. $\sigma(A)$ denotes the set of eigenvalues of the matrix $A$ and $\emptyset$ denotes the empty set. $0_{n\times \nu}\in\mathbb{C}^{n\times\nu}$ is the matrix with all elements equal to $0$. If $n=1$, then $0_{1\times\nu}=[0,...,0]\in\mathbb{R}^{1\times\nu}$ and $0_{\nu}=0_{1\times\nu}^*$.  $\delta(t)$ denotes the Dirac $\delta$-function. The triple $(A,B,C)$ denotes the linear, time-invariant system described by the equations $\dot x=Ax+Bu,\ y=Cx$, with input $u(t)$ from a well defined set of inputs, output $y(t)$ and state $x(t)$.

\section{Preliminaries}\label{sect_prel}

Let $J\in\mathbb{R}^{n\times n}$ be a skew symmetric matrix and $R\in\mathbb{R}^{n\times n}$, $Q\in\mathbb{R}^{n\times n}$ be two symmetric matrices. Consider the single-input, single-output, minimal, port Hamiltonian system
\begin{equation}\label{pH_system} \begin{split}
& \dot x = (J-R)Qx+Bu, \\& y =B^*Qx,
\end{split}\end{equation} where $x(t)\in\mathbb{R}^n$, $u(t)\in\mathbb{R}$, $y(t)\in\mathbb{R}$ and $B\in\mathbb{R}^{n}$. The Hamiltonian is ${\mathcal H}(x)=\frac{1}{2}x^*Qx$. The transfer function of system (\ref{pH_system}) is given by $K:\mathbb{C}\to\mathbb{C},\ K(s)=B^*Q(sI-(J-R)Q)^{-1}B$. %
Let $s_i\in\mathbb{C}$, $i=1,...,\nu$ be such that $s_i\notin\sigma((J-R)Q)$. The moments of (\ref{pH_system}) at $s_i$ are $\eta_0(s_i),\eta_1(s_i),\dots$, with $\eta_k(s_i)=\frac{(-1)^k}{k!}\left.\frac{d^k K(s)}{ds^k}\right|_{s=s_i}$. Note that the moment $\eta_k(s_i)$ represents the $k$-th coefficient of the Taylor expansion of $K(s)$ at $s_i$.
Without loss of generality, throughout the paper we assume that $k=0$.

\subsection{Time-domain moment matching}\label{sect_time}

In this section we give a brief overview of a notion of moment in the time-domain setting, see \cite{astolfi-TAC2010} for a more detailed analysis. Based on this notion, families of parameterized reduced order models are developed.

\noindent Consider the linear system (\ref{pH_system}) and let the matrices $S\in\mathbb{R}^{\nu\times\nu}$ and $L\in\mathbb{R}^{1\times\nu}$, and ${\mathcal Q}\in\mathbb{R}^{\nu\times\nu}$ and ${\mathcal R}\in\mathbb{R}^\nu$ be such that the pair $(L,S)$ is observable and the pair $({\mathcal Q},{\mathcal R})$ is controllable, respectively. Consider the Sylvester equation
\begin{equation}\label{eq_Sylvester}
(J-R)Q\Pi+BL=\Pi S,
\end{equation} in the unknown $\Pi\in\mathbb{C}^{n\times\nu}$ and its dual
\begin{equation}\label{eq_Sylvester_Y}
{\mathcal Q}\Upsilon=\Upsilon(J-R)Q+{\mathcal R} B^*Q,
\end{equation} in the unknown $\Upsilon\in\mathbb{C}^{\nu\times n}$.
Assume that $\sigma(A)\cap\sigma(S)=\emptyset$. Since the system (\ref{pH_system}) is minimal, the Sylvester equation (\ref{eq_Sylvester}) has a unique solution $\Pi$ and ${\rm rank}\ \Pi =\nu$. Similarly, if $\sigma(A)\cap\sigma({\mathcal Q})=\emptyset$, then equation (\ref{eq_Sylvester_Y}) has a unique solution $\Upsilon$ and ${\rm rank}\ \Upsilon=\nu$. (see e.g., \cite{desouza-bhattacharyya-LAA1981}).
\begin{defn}\label{def_PI} \mbox{}\\
\vspace{-0.5cm}
\begin{enumerate}

\item Let $\phi=[\phi_1\ \phi_2\ ...\ \phi_\nu]\in\mathbb{C}^{1\times\nu}$ be such that
\begin{equation*}\label{eq_CPI}
\phi=B^*Q\Pi.
\end{equation*} We call the moments of system (\ref{lin_sys}) at $\sigma(S)$ the elements $\phi_i$, $i=1,...,\nu$. The interpolation points are the eigenvalues of $S$, i.e., $\{s_1,s_2,...,s_\nu\}=\sigma(S)$.

\item Let $\varphi=[\varphi_1\ \varphi_2\ ...\ \varphi_\nu]^*\in\mathbb{C}^{\nu}$ be such that
\begin{equation*}\label{eq_YB}
\varphi=\Upsilon B.
\end{equation*} We call the moments of system (\ref{lin_sys}) at $\sigma(Q)$ the elements $\varphi_i$, $i=1,...,\nu$. The interpolation points are the eigenvalues of ${\mathcal Q}$, i.e., $\{s_1,s_2,...,s_\nu\}=\sigma({\mathcal Q})$.\fin
\end{enumerate}
\end{defn}
The following result establishes the connection between the notions of moment from Definition \ref{def_PI} and the well-defined, steady-state response of the interconnection of (\ref{pH_system}) with a signal generator.
\begin{thm}\label{thm_mom_steady}\cite{astolfi-TAC2010,astolfi-CDC2010} The following statements hold.
\begin{enumerate}

\item Consider system (\ref{pH_system}) and let $S$ be such that $\sigma(S)\subset\mathbb{C}^0$. Assume that $\sigma((J-R)Q)\subset\mathbb{C}^-$. Let
\begin{equation}\label{sig_generator}
\dot \omega = S\omega,\ \omega(0)\ne 0,
\end{equation} with $\omega(t)\in\mathbb{R}^\nu$. Consider the interconnection of systems (\ref{pH_system}) and (\ref{sig_generator}), with $u=L\omega$, where $L$ is such that the pair $(L,S)$ is observable. Then the moments $\phi$ of system (\ref{pH_system}) at $\sigma(S)$ are in one-to-one relation with the well-defined steady-state response of the output $y(t)=B^*Q x(t)$%
\footnote{By one-to-one relation between a set of $\nu$ moments $\eta(s_i)$, $i=1,...,\nu$ and the well-defined steady state response of the signal $y(t)$ we mean that the moments $\eta$ uniquely determine the steady-state response of $y(t)$. \\ Let $(A,B,C)$ be a linear system described by the equations $\dot x=Ax+Bu,\ y=Cx$, with $x(t)\in\mathbb{R}^n$ and $\sigma(A)\in\mathbb{C}^-$. Then $x(t)=x_p(t)+x_t(t)$, with $\displaystyle x_p(t)=\lim_{t_0\to -\infty}\int_{t_0}^t e^{A(t-\tau)}B u(\tau) d\tau$. $x_t(t)$ is the transient component of the state $x(t)$, i.e., $\displaystyle \lim_{t\to\infty} x_t(t)=0$ and $x_p(t)$ is the steady-state. Consequently, $y_p(t)=Cx_p(t)$ is the steady-state response of the linear system $(A,B,C)$.}%
of such interconnected system.

\item Consider the system (\ref{pH_system}) and let $\mathcal Q$ be such that $\sigma({\mathcal Q})\subset\mathbb{C}^0$. Consider the system
\begin{equation}\label{signal_gen_M} \begin{array}{ll}
\dot\omega = {\mathcal Q}\omega+{\mathcal R} v, \\ d=\omega+\Upsilon x,
\end{array}\end{equation}  where $\omega(t)\in\mathbb{R}^\nu,\ d(t)\in\mathbb{R}^\nu,\ v(t)\in\mathbb{R}$ and $\mathcal R$ is such that the pair $({\mathcal Q},{\mathcal R})$ is controllable. Consider the interconnection between the system (\ref{pH_system}) and the system (\ref{signal_gen_M}), with $v=y$. Assume that $\sigma((J-R)Q)\subset\mathbb{C}^{-}$, $x(0)=0$, $\omega(0)=0$ and $u(t)=\delta(t)$. Then the moments $\varphi$ of system (\ref{pH_system}) at $\sigma({\mathcal Q})$ are in one-to-one relation with the well-defined steady-state response of the output $d(t)$.\fin

\end{enumerate}
\end{thm}

Based on Definition \ref{def_PI}, we define families of parameterized models of order $\nu$ that match the moments of (\ref{pH_system}) at the interpolation points $\{s_1,...,s_\nu\}=\sigma(S)$ and  families of parameterized models of order $\nu$ that match the moments of (\ref{pH_system}) at $\{s_1,...,s_\nu\}=\sigma({\mathcal Q})$.

\begin{thm}\cite{astolfi-TAC2010,astolfi-CDC2010}\label{thm_redmod_CPi} The following statements hold.
\begin{enumerate}

\item Let the pair $(L,S)$ be observable and assume $\sigma(A)\cap\sigma(S)=\emptyset$. Let $\xi(t)\in\mathbb{R}^\nu$ and consider the family of linear models
\begin{equation}\label{model_gen_G} \Sigma_G:\ \left\{\begin{array}{ll}
\dot \xi=(S-GL)\xi+Gu, \\ \psi=B^*Q\Pi\xi,
\end{array}\right.\end{equation} parameterized in $G\in\mathbb{C}^\nu$, where $\Pi$ is the unique solution of (\ref{eq_Sylvester}). Assume $\sigma(S-GL)\cap\sigma(S)=\emptyset$. Let $\widehat\phi\in\mathbb{C}^{1\times \nu}$ be the moments of (\ref{model_gen_G}) at $\sigma(S)$. Then (\ref{model_gen_G}) describes a family of reduced order models of (\ref{lin_sys}), parameterized in $G$, achieving moment matching at $\sigma(S)$, i.e., $\phi=\widehat\phi$.

\item Let the pair $({\mathcal Q},{\mathcal R})$ be controllable and assume $\sigma(A)\cap\sigma({\mathcal Q})=\emptyset$. Let $\xi(t)\in\mathbb{R}^\nu$ and consider the family of linear models
\begin{equation}\label{model_gen_H} \Sigma_H:\ \left\{\begin{array}{ll}
\dot \xi=({\mathcal Q}-{\mathcal R}H)\xi+\Upsilon Bu, \\ \psi=H\xi,
\end{array}\right.\end{equation} parameterized in $H\in\mathbb{R}^{1\times\nu}$, where $\Upsilon$ is the unique solution of (\ref{eq_Sylvester_Y}). Assume $\sigma({\mathcal Q}-{\mathcal R}H)\cap\sigma({\mathcal Q})=\emptyset$. Let $\widehat\varphi\in\mathbb{C}^{1\times \nu}$ be the moments of (\ref{model_gen_H}) at $\sigma({\mathcal Q})$. Then (\ref{model_gen_H}) describes a family of reduced order models of (\ref{lin_sys}), parameterized in $H$, achieving moment matching at $\sigma({\mathcal Q})$, i.e., $\varphi=\widehat\varphi$.
\fin
\end{enumerate}
\end{thm}

The next result shows that the models $\Sigma_G$ and $\Sigma_H$ achieve moment matching at $\sigma(S)$ and $\sigma({\mathcal Q})$ for all parameters $G$ and $H$, respectively.
\begin{prop}\label{prop_matching_happens}\cite{astolfi-TAC2010,astolfi-CDC2010} The following statements hold.
\begin{enumerate}
\item Consider the family of systems (\ref{model_gen_G}). Consider a $\nu$-th order model of system (\ref{pH_system}) at $\sigma(S)$ and let $\widehat K(s)$ be its transfer function. Then there exists a unique $G$ such that $\widehat K(s)=B^*Q\Pi(sI-S+GL)^{-1}G$.

\item Consider system (\ref{pH_system}) and the family of $\nu$-th order models (\ref{model_gen_H}). 
Let $\Upsilon$ be the (unique) solution of equation (\ref{eq_Sylvester_Y}). Then, for all $H$, any model in the family (\ref{model_gen_H}) is a reduced order model of system (\ref{pH_system}) achieving moment matching at $\sigma({\mathcal Q})$. \fin
\end{enumerate}
\end{prop}
Note that the results so far cannot be applied to the case of infinite interpolation points, i.e., do not apply to $s=\infty$.

\subsection{The MIMO case}\label{sect_MIMO_prel}

Consider a MIMO system of the form (\ref{pH_system}), with the input $u(t)\in\mathbb{R}^m$,  i.e., $B\in\mathbb{C}^{n\times m}$ and the transfer function $K(s)\in\mathbb{C}^{m\times m}$.  Let $S\in\mathbb{C}^{\nu\times\nu}$ and $L=[l_1\ l_2\ ...\ l_\nu]\in\mathbb{C}^{m\times\nu}$, $l_i\in\mathbb{C}^{m}$, $i=1,...,\nu$, be such that the pair $(L,S)$ is observable. Let $\Pi\in\mathbb{C}^{n\times\nu}$ be the unique solution of the Sylvester equation (\ref{eq_Sylvester}). Simple computations yield that the moments $\eta(s_i)=K(s_i)l_i$, $\eta(s_i)\in\mathbb{C}^p$, $i=1,...,\nu$ of system (\ref{pH_system}) at $\{s_1,...,s_\nu\}=\sigma(S)$ are in one-to-one relation with $C\Pi$. Consider the following system
\begin{equation}\label{red_mod_FGH_finite}\begin{split}
&\dot \xi=F\xi+Gu,\\& \psi=H\xi,
\end{split}\end{equation}
with $\xi(t)\in\mathbb{R}^\nu$, $\psi(t)\in\mathbb{R}^p$, $G\in\mathbb{C}^{\nu\times m}$ and  $H\in\mathbb{C}^{p\times\nu}$. The model reduction problem for MIMO systems boils down to finding a $\nu$-th order model described by the equations (\ref{red_mod_FGH_finite}) which satisfies the conditions
\begin{equation}\label{eq_tangent}
K(s_i)l_i=\widehat K(s_i)l_i,\ i=1,...,\nu,
\end{equation}
where $\widehat K(s)=H(sI-F)^{-1}G$ is the transfer function of (\ref{red_mod_FGH_finite}). The relations (\ref{eq_tangent}) are called the right tangential interpolation conditions, see \cite{gallivan-vandendorpe-vandooren-SIAM2004}. It immediately follows that the solution to this problem is provided by a direct application of Theorem \ref{thm_redmod_CPi}, i.e., a class of reduced order MIMO models that achieve moment matching in the sense of satisfying the tangential interpolation conditions (\ref{eq_tangent}) is given by $\Sigma_G=(S-GL,G,B^*Q\Pi)$ as in (\ref{model_gen_G}).

Similarly, we may define the left tangential interpolation problem and its solution. To this end, let ${\mathcal Q}\in\mathbb{C}^{\nu\times\nu}$ and ${\mathcal R}=[r_1^*\ ...\ r_\nu^*]^*\in\mathbb{C}^{\nu\times p}$, $r_i\in\mathbb{C}^{1\times p}$, $i=1,...,\nu$, be such that the pair $({\mathcal Q},{\mathcal R})$ is controllable. Let $\Upsilon\in\mathbb{C}^{\nu\times n}$ be the unique solution of (\ref{eq_Sylvester_Y}). Hence the moments $\eta(s_i)=r_iK(s_i)$, $\eta(s_i)\in\mathbb{C}^{1\times m}$, $i=1,...,\nu$, of system (\ref{pH_system}) at $\{s_1,...,s_\nu\}=\sigma({\mathcal Q})$ are in one-to-one relation with $\Upsilon B$. The model reduction problem boils down to finding a $\nu$-th order model described by the equations (\ref{red_mod_FGH_finite}) which satisfies the conditions
\begin{equation}\label{eq_tangent_dual}
r_iK(s_i)=r_i\widehat K(s_i),\ i=1,...,\nu.
\end{equation}
The relations (\ref{eq_tangent_dual}) are called the left tangential interpolation conditions, see \cite{gallivan-vandendorpe-vandooren-SIAM2004} and the solution to this problem is provided by a direct application of Theorem \ref{thm_redmod_CPi}, i.e., a class of reduced order MIMO models that achieve moment matching in the sense of satisfying the tangential interpolation conditions (\ref{eq_tangent_dual}) is given by $\Sigma_H=({\mathcal Q}-{\mathcal R}H,\Upsilon B,H)$ as in (\ref{model_gen_H}).

Note finally that also in the MIMO case the models are parameterized in $L$ and $\mathcal R$, respectively. Their choice is important in establishing appropriate directions for interpolation. Throughout the rest of the paper we discuss the SISO case, i.e., $m=p=1$, the results being easily extended to tangential interpolation for MIMO systems. However, when necessary, we make specific remarks about the latter case.

\subsection{Krylov projections}\label{sect_krylov}

In this section we give a brief overview on the port Hamiltonian reduced order models obtained using Krylov projections.

\begin{prop}\label{prop_red_PH_V}\cite{lohmann-wolf-eid-kotyczka-TRAC2009,wolf-lohmann-eid-kotyczka-EJC2010}
Consider the system (\ref{pH_system}). Let $s_0\in\mathbb{C}$ be such that $s_0\notin\sigma((J-R)Q)$ and let $V\in\mathbb{R}^{n\times \nu}$ be such that $V\in\ {\rm span} \{((J-R)Q-s_0I)^{-1} B,...,((J-R)Q-s_0I)^{-\nu} B\}$. Consider the port Hamiltonian system
\begin{equation}\label{model_pH_WV} \Sigma_{V}: \left\{\begin{split}
&\dot \xi= (J_r-R_r)Q_r\xi+B_ru,\\& \psi=B_r^*Q_r\xi,
\end{split}\right.\end{equation} with
\begin{equation}\label{model_pH_WV_paramters}\begin{split}
&J_r=V^*QJQV,\ R_r=V^*QRQV, \\& Q_r=(V^*QV)^{-1},\ B_r=V^*QB.
\end{split}\end{equation}
Then $\Sigma_V$ matches the first $\nu$ moments of (\ref{pH_system}) at $s_0$.
\fin\end{prop} %

Furthermore, in \cite{gugercin-polyuga-beattie-vdschaft-ARXIV2011} a Krylov projection technique is proposed for the computation of a port Hamiltonian reduced order model that achieves moment matching, i.e., tangential interpolation, in the MIMO case.
\begin{thm}\label{thm_red_PH_WV}\cite{gugercin-polyuga-beattie-vdschaft-ARXIV2011}
Consider a MIMO system (\ref{pH_system}) with the input $u(t)\in\mathbb{R}^m$ and the output $y(t)\in\mathbb{R}^m$, i.e. $B\in\mathbb{C}^{n\times m}$. Let $s_1,...,s_\nu\in\mathbb{C}$ be such that $s_1\ne s_2\ne ... \ne s_\nu$. Let $l_1,...,l_\nu\in\mathbb{C}^m$ and
\begin{equation}\label{eq_V}
V=[(s_1 I-(J-Q)R)^{-1}Bl_1\ ...\ (s_\nu I-(J-Q)R)^{-1}Bl_\nu].
\end{equation}
Let $M\in\mathbb{C}^{\nu\times\nu}$ be any nonsingular matrix such that $\widehat V=VM$ is real. Define $\widehat W=Q\widehat V(\widehat V^*Q\widehat V)^{-1}$ and
\begin{equation}\label{model_pH_W_paramters}\begin{split}
&J_r=\widehat W^*J\widehat W,\ R_r=\widehat W^*R\widehat W, \\& Q_r=(\widehat V^*Q\widehat V)^{-1},\ B_r=\widehat W^*B.
\end{split}\end{equation}
Then the reduced order model $\Sigma_{\widehat V}$ as in (\ref{model_pH_WV}) is port Hamiltonian, passive and satisfies the right tangential interpolation conditions (\ref{eq_tangent}).\fin
\end{thm}
Note that Theorem \ref{thm_red_PH_WV} yields a class of reduced order port Hamiltonian models, with state-space realizations parameterized in $M$, which achieve moment matching. Results which extend this technique to matching a number of moments which is twice the number of selected interpolation points can be found in \cite{polyuga-vdschaft-TAC2011}.
\begin{rem}\label{obs_optimal_choice}\n
To the best of our knowledge, the existing moment matching based model order reduction techniques do not yield approximation error bounds. However, a choice of interpolation points that improves the accuracy of the approximation, i.e., stems from solving an $H_2$ optimal approximation problem, is at the mirror images of the poles of the approximant. In detail, if $\widehat K(s)$ is the approximation that yields the best approximation of $K(s)$ associated to (\ref{pH_system}), in the $H_2$ norm, then $\widehat K(s)$ satisfies the conditions
\begin{subequations}\label{eq_H2opt}
\begin{align}
\widehat K(-\widehat\lambda_i) &= K(-\widehat\lambda_i), \label{eq_H2opt1}\\ \nonumber\\
\left.\frac{d\widehat K(s)}{ds}\right|_{s=-\widehat\lambda_i} &= \left.\frac{dK(s)}{ds}\right|_{s=-\widehat\lambda_i}, \label{eq_H2opt2}
\end{align}
\end{subequations}
with  $i=1,...,\nu$, where $\lambda_i$ are the poles of $\widehat K(s)$, see e.g., \cite{gugercin-antoulas-beattie-SIAM2008,meier-luenberger-TAC1967,gugercin-antoulas-CDC2003}. Hence, an accurate approximation of a system (\ref{pH_system}), in the sense that the $H_2$-norm of the approximation error decreases, is obtained by finding a reduced order model that achieves moment matching in the sense of relations (\ref{eq_H2opt}). Since the poles of the approximation are not known in advance, a suitable choice of interpolation points (see \cite{gugercin-antoulas-CDC2003,gugercin-antoulas-LINALG2006}) is
\begin{equation}\label{eq_H2opt_proposed}
s_i=-\lambda_i,\ i=1,...,\nu,
\end{equation}
where $\lambda_i\in\sigma((J-R)Q)$, with highest residues of the transfer function $K(s)$ .\fin
\end{rem}

Proposition \ref{prop_red_PH_V} directly applies to Markov parameter matching, i.e., matching at infinity. However, this is not the case for Theorem \ref{thm_red_PH_WV}. This is due to the fact that the construction and direct application of the projectors  $W$ and $V$ does not yield a reduced order model with port Hamiltonian structure. However, in \cite{polyuga-vdschaft-ECC2009,polyuga-vdschaft-AUT2010} a solution is obtained in a specific set of coordinates, i.e., performing a coordinate transformation the port Hamiltonian system is brought in a form where Proposition \ref{prop_red_PH_V} applies.
A precise definition of Markov parameters and the corresponding detailed arguments are found in Section \ref{sect_MarkovM}.

To the best of the authors' knowledge there is no structured way of determining the number $\nu$ of interpolation points required for approximation. However, a fair choice appears to be the number of large Hankel singular values of the given system, see, e.g., \cite{ionescu-scherpen-iftime-astolfi-MTNS2012}.

\section{Markov parameter matching as moment matching for a class of descriptor systems}\label{sect_MarkovM}

As mentioned in Section \ref{sect_time}, the time-domain moment matching tools do not directly apply to $s=\infty$. In this section we present a time-domain approach to moment matching at $s=\infty$, i.e., Markov parameter matching. The idea is to perform the change of variable $\tau=1/s \Leftrightarrow s=1/\tau$ and turn the problem of Markov parameter matching into a moment matching problem at $\tau=0$. First we discuss the problem of moment matching at any $\tau\in\mathbb{C}$ and provide time-domain interpretations of the notion of moment and moment matching. Based on this notion of moment, we characterize the classes of parameterized reduced order models that match the moments of a given system at $\tau$ and in particular at $\tau=0$.

Consider a linear, minimal system described by the equations
\begin{equation}\label{lin_sys}
\begin{split} & \dot x=Ax+Bu, \\& y=Cx, \end{split}
\end{equation} with $x(t)\in\mathbb{R}^{n},\ y(t)\in\mathbb{R},\ u(t)\in\mathbb{R}$. Let $K(s)=C(sI-A)^{-1}B$ be its transfer function. The first $\nu +1$ Markov parameters are the coefficients of the series expansion of $K(s)$ around $s=\infty$, i.e., they are the first $\nu +1$ moments of $K(s)$ at $\infty$, namely
\begin{equation}\label{Markov_par}
\eta_0(\infty)=0,\ \eta_k(\infty)=CA^{k-1}B,\ k=1,...,\nu.
\end{equation} Let $\tau\in\mathbb{C}$ and note that the matrix pencil $I-A\tau$ is regular, i.e., $\det(I-A\tau)\ne 0$, for some $\tau$ (see e.g., \cite{chu-LAA1987}). According to, e.g., \cite{campbell-1980}, $(I-A\tau)^{-1}$ exists and so the function $\widetilde K(\tau)=K\left(\frac{1}{\tau}\right)=C(I-A\tau)^{-1}B\tau$ is well-defined. Furthermore, we have that $\frac{d^{k+1} \widetilde K(\tau)}{d\tau^{k+1}} =(k+1)!C[(I-A\tau)^{-k-1}A^{k}B+(I-A\tau)^{-k-2}A^{k+1}B\tau]$, yielding
\begin{equation*}\frac{1}{(k+1)!}\frac{d^{k+1} \widetilde K(\tau)}{d\tau^{k+1}} =C(I-A\tau)^{-k-2}A^{k}B.\end{equation*} The moments of $\widetilde K(\tau)$ at $\tau=\tau^*\in\mathbb{C}$ are given by
\begin{equation}\label{eq_mom_tilde} \widetilde\eta_k(\tau^*) =\frac{1}{(k+1)!} \left.\frac{d^{k+1} \widetilde K(\tau)}{d\tau^{k+1}}\right|_{\tau=\tau^*}\end{equation} and the moments $\eta_0(\infty),...,\eta_\nu(\infty)$ are given by
\begin{equation*}
\eta_k(\infty)=\frac{1}{(k+1)!}\left.\frac{d^{k+1} \widetilde K(\tau)}{d\tau^{k+1}}\right|_{\tau=0}=\widetilde\eta_k(0).
\end{equation*} We now consider the following moment matching problem. Given the function $\widetilde K(\tau)$ and the point $\tau^*\in\mathbb{C}$ find $\widehat K(\tau)$ such that the first $\nu+1$ moments at $\tau^*$ match, i.e. $\widetilde\eta_k(\tau)=\frac{d^{k+1} \widetilde K}{d\tau^{k+1}}(\tau^*)=\frac{d^{k+1} \widehat K}{d\tau^{k+1}}(\tau^*)$, for all $k=0,...,\nu$. In particular, we are interested in the case $\tau^*=0$, which solves the Markov parameter matching problem.
\begin{prop}\label{prop_mom_Markov}
Consider the system (\ref{lin_sys}) and $\tau^*\in\mathbb{C}$. Let
\begin{equation*}\begin{split}& L=[1\ 0\ 0\ \dots\ 0]\in\mathbb{R}^{1\times (\nu+1)}, \\& \ S=\left[\begin{array}{ccccc}\tau^* &1 &0 & \dots &0 \\ 0& \tau^* &1 &\dots &0 \\ \vdots & \vdots & \ddots & \ddots & \vdots \\ 0 & \dots &0 &\tau^* &1 \\ 0 &\dots &\dots &0 &\tau^* \end{array}\right]\in\mathbb{C}^{(\nu+1)\times(\nu+1)}.\end{split}\end{equation*} Assume that $\lambda\tau^*\ne 1$, for any $\lambda\in\sigma(A)$.
\begin{enumerate}

\item Let $\Pi\in\mathbb{C}^{n\times (\nu+1)}$ be the unique solution of the Sylvester equation
\begin{equation}\label{eq_matrix_Pi_Markov}
A\Pi S+BL=\Pi.
\end{equation} Then the moments $\widetilde\eta_0(\tau^*),\ \dots,\ \widetilde\eta_\nu(\tau^*)$ of $\widetilde K(\tau)=K(1/\tau)$ satisfy
\begin{equation*}\label{eq_mom_Markov}
[\widetilde\eta_0(\tau^*)\ \widetilde\eta_1(\tau^*)\ \dots\ \widetilde\eta_\nu(\tau^*)]=C\Pi S.
\end{equation*}

\item Let $\bar\Pi\in\mathbb{C}^{n\times (\nu+1)}$ be the unique solution of the Sylvester equation
\begin{equation}\label{eq_matrix_Pi_Markov_bar}
A\bar\Pi S+BLS=\bar\Pi.
\end{equation} Then the moments $\widetilde\eta_0(\tau^*),\ \dots,\ \widetilde\eta_\nu(\tau^*)$ of $\widetilde K(\tau)=K(1/\tau)$ satisfy
\begin{equation*}\label{eq_mom_Markov_bar}
[\widetilde\eta_0(\tau^*)\ \widetilde\eta_1(\tau^*)\ \dots\ \widetilde\eta_\nu(\tau^*)]=C\bar\Pi.
\end{equation*} \vskip -0.3cm \fin

\end{enumerate}
\end{prop}
\begin{proof} The assumption that $\lambda\tau^*\ne 1$, for any $\lambda\in\sigma(A)$ and the regularity of the matrix pencil $I-A\tau$ yield that the solutions $\Pi$ and $\bar\Pi$ of the Sylvester equations (\ref{eq_matrix_Pi_Markov}) and (\ref{eq_matrix_Pi_Markov_bar}), respectively, exist and are unique (see e.g., \cite{chu-LAA1987}).
\\
To prove the statement (1), let $\Pi=[\Pi_0\ \Pi_1\ \dots\ \Pi_\nu]$. By (\ref{eq_matrix_Pi_Markov}) we have the sequence of equalities
\begin{align*}
A\Pi_0\tau^*+B=\Pi_0 & \Leftrightarrow \Pi_0=(I-A\tau^*)^{-1}B, 
\\
A\Pi_0+A\Pi_1\tau^*=\Pi_1 &\Leftrightarrow \Pi_1 =(I-A\tau^*)^{-1}A\Pi_0 \\  & \hspace{0.96cm} =(I-A\tau^*)^{-1}A(I-A\tau^*)^{-1}B \\ & \hspace{0.96cm} =(I-A\tau^*)^{-2}AB, 
\\
&\hskip 0.22cm\vdots 
\\
A\Pi_{\nu-1}+A\Pi_\nu \tau^*=\Pi_\nu &\Leftrightarrow \Pi_\nu=(I-A\tau^*)^{-1}\Pi_{\nu-1} \\ & \hspace{0.97cm} =(I-A\tau^*)^{-\nu-1}A^\nu B, 
\end{align*}
from which the claim follows. To prove the second statement, let $\bar\Pi=[\bar\Pi_0\ \bar\Pi_1\ \dots\ \bar\Pi_\nu]$. By (\ref{eq_matrix_Pi_Markov_bar}) we have the sequence of equalities
\begin{align*}
A\bar\Pi_0\tau^*+B\tau^*=\bar\Pi_0 & \Leftrightarrow \bar\Pi_0=(I-A\tau^*)^{-1}B\tau^*, 
\\
A\bar\Pi_0+A\bar\Pi_1\tau^*+B=\bar\Pi_1 &\Leftrightarrow \bar\Pi_1=(I-A\tau^*)^{-1}(A\bar\Pi_0+B) \\& \hspace{0.96cm} =(I-A\tau^*)^{-2}A^2B\tau^* \\ & \hspace{1.3cm} +(I-A\tau^*)^{-1}B, 
\\
&\hskip 0.22cm \vdots 
\\
A\bar\Pi_{\nu-1}+A\bar\Pi_\nu \tau^*=\bar\Pi_\nu &\Leftrightarrow \bar\Pi_\nu=(I-A\tau^*)^{-(\nu+1)}A^\nu B\tau^* \\ & \hspace{1.3cm} +(I-A\tau^*)^{-\nu}A^{\nu-1}B, 
\end{align*}
from which the claim follows. \end{proof}
Proposition \ref{prop_mom_Markov} can be extended to the general case of any (non-derogatory\footnote{A matrix is non-derogatory if its characteristic and minimal polynomials coincide.}\hspace{-0.08cm}) matrix $S$ that satisfy the assumption made on $\sigma(A)$. Hence, throughout the rest of the section we make the following standing assumption.
\begin{assum}\n\label{ass_eigenvalues_Markov}
$\lambda\mu\ne 1$, for any $\lambda\in\sigma(A)$ and $\mu\in\sigma(S)$.
\end{assum}
\begin{exmp}\n\label{example_markov_ABC}

\begin{figure}[h]\centering
\subfigure[]{\psfrag{U}{$u$} \psfrag{S}{\large $\frac{1}{s}$} \psfrag{X2}{$x_2$}
\psfrag{X1}{\hspace{0.4cm}$x_1$} \psfrag{A}{$a$} \psfrag{B}{$b$}
\includegraphics[scale=0.6]{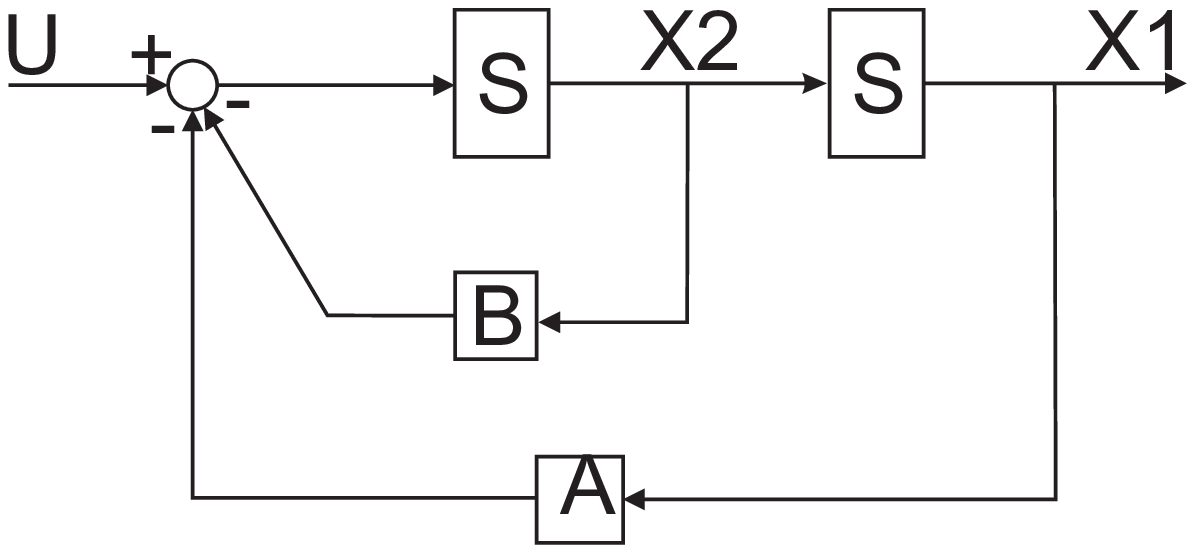} \label{fig_example_markov}}
\hspace{\subfigtopskip}
\hspace{\subfigbottomskip}
\subfigure[]{\psfrag{U}{$u$} \psfrag{S}{\large $\tau$} \psfrag{X2}{$x_2$}
\psfrag{X1}{\hspace{0.4cm}$x_1$} \psfrag{A}{$a$} \psfrag{B}{$b$}
\includegraphics[scale=0.6]{example_markov_ABC.EPS}\label{fig_example_markov_1}}
\caption{Block diagrams of systems (\ref{exmp_markov_ABC}) (a) and (\ref{exmp_markov_ABC_1}) (b).}\end{figure}

\noindent Consider the system depicted in Fig. \ref{fig_example_markov}, described by the equations
\begin{align}\label{exmp_markov_ABC}
 \dot x_1 &= x_2, \nonumber\\ \dot x_2 &= -ax_1-bx_2+u, \\ y &= x_1, \nonumber
\end{align} with the transfer function $K(s)=\frac{1}{s^2+bs+a}$.

Consider now the system depicted in Fig. \ref{fig_example_markov_1}, where we have replaced the integrators with differentiators, i.e., $\tau=\frac{1}{s}$. Note that $x_1(t)=\frac{d}{dt}(x_2(t))$ and $x_2(t)=\frac{d}{dt}(-ax_1(t)-bx_2(t)+u(t))$ yield the {\it descriptor state-space realization}
\begin{align}\label{exmp_markov_ABC_1}
\dot x_2 & =  x_1,\nonumber \\ -a\dot x_1-b\dot x_2 & =  x_2-\dot u, \\ y & =  x_1, \nonumber
\end{align} with the transfer function $\widetilde K(\tau)=K\left(\frac{1}{\tau}\right)=\frac{\tau^2}{a\tau^2+b\tau+1}$. Hence, the Markov parameters of (\ref{exmp_markov_ABC}) are the moments of the system (\ref{exmp_markov_ABC_1}) at $\tau=0$.
\fin\end{exmp}

\begin{figure}[h]\centering
\psfrag{Sys}{\small \hspace{-0.15cm} $\begin{array}{ll} A\dot x=x-B\dot u \\ y=Cx\end{array}$} \psfrag{Sig}{\scriptsize \hspace{-0.1cm} $\begin{array}{ll}\dot \omega=S\omega\\ \theta=L\omega\end{array}$}
\psfrag{XT}{$\theta=u$} \psfrag{Y}{$y$}
\includegraphics[scale=0.8]{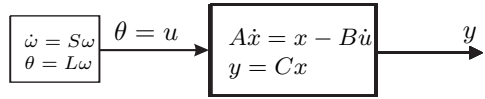}\caption{Interconnection of the signal generator and the system (\ref{sys_descript_2}), with transfer function $\widetilde K(\tau)$.} \label{diagram_omega_block_L}
\end{figure}

The following result provides a time-domain interpretation of the moments of $\widetilde K(\tau)=K(1/\tau)$, where $K$ is the transfer function of the system (\ref{lin_sys}). Namely, the moments of $\widetilde K(\tau)$ are related to the well-defined steady-state response of the output of a descriptor state-space representation associated to $\widetilde K(\tau)$, with the input given by a signal generator defined by the interpolation points given by the set $\sigma(S)$ (see Fig. \ref{diagram_omega_block_L}).

\begin{thm}\label{thm_Markov_time}
Consider the system (\ref{lin_sys}) and let $S$ be such that $\sigma(S)\subset\mathbb{C}^0$. Assume that $\sigma(A)\subset\mathbb{C}^-$. Let
\begin{equation}\label{sig_generator_descript}
\dot \omega = S\omega,\ \omega(0)\ne 0,
\end{equation} with $\omega(t)\in\mathbb{R}^\nu$.
\begin{description}

\item[{\bf i.}] Consider the descriptor state-space representation associated to the transfer function $\widetilde K$ given by
\begin{equation}\label{sys_descript_1}\begin{split}
& A\dot x=x-Bu, \\& y=C\dot x,
\end{split}\end{equation}
with $x(t)\in\mathbb{R}^n$, $u(t)\in\mathbb{R}$ and $y(t)\in\mathbb{R}$. Then the moments $\widetilde\eta_k(\tau_i)$, $k=0,1,2,...$, $i=0,...,\nu$, of the system (\ref{lin_sys}) at $\sigma(S)$, are in one-to-one relation with the well-defined steady-state response of the output of the interconnection of systems (\ref{sys_descript_1}) and (\ref{sig_generator_descript}), with $u=L\omega$, where $L$ is such that the pair $(L,S)$ is observable.
\vspace{0.1cm}
\item[{\bf ii.}] Consider the descriptor state-space representation associated to the transfer function $\widetilde K$ given by
\begin{equation}\label{sys_descript_2}\begin{split}
& A\dot x=x-B\dot u, \\& y=Cx,
\end{split}\end{equation}
with $x(t)\in\mathbb{R}^n$, $u(t)\in\mathbb{R}$ and $y(t)\in\mathbb{R}$. Then the moments $\widetilde\eta_k(\tau_i)$, $k=0,1,2,...$, $i=0,...,\nu$, of the system (\ref{lin_sys}) at $\sigma(S)$, are in one-to-one relation with the well-defined steady-state response of the output of the interconnection of systems (\ref{sys_descript_2}) and (\ref{sig_generator_descript}), with $u=L\omega$, where $L$ is such that the pair $(L,S)$ is observable.\fin
\end{description}
\end{thm}
\begin{proof}
By the interconnection between (\ref{sys_descript_1}) and (\ref{sig_generator_descript}), the first equation in (\ref{sys_descript_1}) yields $x=A\dot x+BL\omega=A\dot x+(\Pi-A\Pi S)\omega$, where $\Pi$ is the unique solution of (\ref{eq_matrix_Pi_Markov}). This further yields $x-\Pi\omega=A(\dot x-\Pi\dot\omega)$. The output $y(t)$ satisfies $y(t)=C(\dot x-\Pi\dot \omega)+C\Pi\dot\omega$. The assumptions $\sigma(A)\subset\mathbb{C}^{-}$ and $\sigma(S)\subset\mathbb{C}^0$ yield $\displaystyle\lim_{t\to\infty}[x(t)-\Pi\omega(t)]=0$, implying that the first term of $y(t)$ vanishes. Hence the steady-state response exists and is well-defined and equal to $C\Pi\dot\omega=C\Pi S\omega$, which proves the claim {\bf i.} Since $\sigma(A)\subset\mathbb{C}^{-}$, the statement {\bf ii.} is a direct application of Theorem \ref{thm_mom_steady}, with the signal $u(t)$ replaced by $\dot u(t)$.
\end{proof}

\begin{cor}\label{cor_Markov_time}
Consider the system (\ref{lin_sys}) satisfying the assumptions in Theorem \ref{thm_Markov_time}. Then the first $\nu$ Markov parameters of (\ref{lin_sys}) are in a one-to-one relationship with the well-defined steady-state response of the system (\ref{sys_descript_1}) (or (\ref{sys_descript_2})) to the input given by a signal generator defined by the interpolation points $\tau^*=0$ with multiplicity $\nu$.
\fin\end{cor}

We are now ready to define the notion of reduced order models that match the moments of $\widetilde K(\tau)=K(1/\tau)$.
\\
Consider an observable pair $(L,S)$ and the system
\begin{equation}\label{red_mod_FGH} \Sigma_{\mathbf{\Pi}}: \left\{
\begin{split} & \dot\xi = F\xi+Gu, \\& \psi=H\xi,\end{split} \right.
\end{equation} with $\xi(t)\in\mathbb{R}^\nu$ and $\mathbf{\Pi}=\Pi$, or $\mathbf{\Pi}=\bar\Pi$. Let $K_{\mathbf{\Pi}}(s)$ be the transfer function of system (\ref{red_mod_FGH}) and let $\widetilde K_{\mathbf{\Pi}}(\tau)=K_{\mathbf{\Pi}}(1/\tau)$. Then, if $\mathbf{\Pi}=\Pi$ and $\lambda\mu\ne 1$, for any $\lambda\in\sigma(F)$ and $\mu\in\sigma(S)$, (\ref{red_mod_FGH}) is a reduced order model that matches the first $\nu$ moments of $\widetilde K(\tau)$ at $\sigma(S)$ if there exists an invertible matrix $P\in\mathbb{R}^{\nu\times\nu}$ such that
\begin{equation}\label{cond_FGH}
C\Pi S= HPS,\ FPS+GL=P.
\end{equation} If $\tau^*=0$, then the Markov parameters of $\widetilde K_{\Pi}(\tau)$ match the first $\nu$ Markov parameters of $K(s)$.

Finally, for $\mathbf{\Pi}=\bar \Pi$, the system (\ref{red_mod_FGH}) is a reduced order model that matches the first $\nu$ moments of $\widetilde K(\tau)$ at $\sigma(S)$ if there exists an invertible matrix $\bar P\in\mathbb{R}^{\nu\times\nu}$ such that
\begin{equation}\label{cond_FGH_1}
C\Pi= H\bar P,\ F\bar PS+GLS=\bar P.
\end{equation}
\begin{rem}\n
Assume $S$ is invertible, i.e. $\tau^*\ne 0$. Let $P=I$. A reduced order model that matches the moments of $\widetilde K(\tau)$ at $\tau^*$ is given by equations (\ref{red_mod_FGH}) with $F=(I-GL)S^{-1}$ and $H=C\Pi$. Furthermore, if $\bar P=I$, according to (\ref{cond_FGH_1}) another reduced order model that matches the moments of $\widetilde K(\tau)$ at $\tau^*$ is given by equations (\ref{red_mod_FGH}) with $F=S^{-1}-GL$ and $H=C\bar \Pi$.
\fin\end{rem}
\begin{exmp}\label{example_markov_match}\n
Let
\begin{equation*}
L=[1\ 0\ 0],\ S=\left[\begin{array}{ccc} 0 &1 &0 \\ 0 &0 &1 \\ 0 &0 &0\end{array}\right],\ C\Pi=[\eta_1\ \eta_2\ \eta_3].
\end{equation*} The first three Markov parameters of $K(s)$ are $C\Pi S=[0\ \eta_1\ \eta_2]=[0\ CB\ CAB]$. A reduced order model that matches these Markov parameters is given by equations (\ref{cond_FGH}), with
$$F=\left[\begin{array}{ccc} 0 & 0 & 1 \\ f_{21} & 0 & 0 \\ f_{31} & 1 & f_{33} \end{array}\right],\ G=[1\ 0\ 0]^*,\ H=[\eta_1\ \eta_2\ \eta_3].$$
\noindent In general, let $0_{1\times \nu}= [0\ \dots\ 0]\in\mathbb{R}^{1\times \nu}$, $S_1=[1\ \dots\ 0]\in\mathbb{R}^{1\times \nu}$ and
$$S_2=\left[\begin{array}{ccccc}\tau^* &1 &0 & \dots &0 \\ 0& \tau^* &1 &\dots &0 \\ \vdots & \vdots & \ddots & \ddots & \vdots \\ 0 & \dots &0 &\tau^* &1 \\ 0 &\dots &\dots &0 &\tau^* \end{array}\right]\in\mathbb{R}^{\nu\times\nu}.$$
\noindent Then $L=[1\ 0_{1\times \nu}]$ and $S=\left[\begin{array}{cc} 0 & S_1 \\ 0_{\nu} & S_2\end{array}\right]$ are such that the pair $(L,S)$ is observable. Furthermore, let $C\Pi=[\eta_1\ \eta_2\ \dots \ \eta_\nu\ \eta_{\nu+1}]\in\mathbb{R}^{1\times(\nu+1)}$.
The first $\nu+1$ Markov parameters are 

$$C\Pi S=[0\ \eta_1\ \dots\ \eta_{\nu}]=[0\quad CB\quad CAB\quad \dots\quad CA^{\nu-1}B].$$ 

Let $F=\left[\begin{array}{cc} F^*_{11} & F_{12} \\ F_{21} & F_{22} \end{array}\right]$, $G=[G_1\ G_2^*]^*$ and $P=I$. Solving (\ref{cond_FGH}) yields a family of reduced order models of dimension $\nu+1$ that achieve Markov parameters matching characterized by $G_1=1,\ G_2=0$, $F_{11}=0$, $F_{12}=[0\ \dots\ 1]$, $F_{12}$, $H=[\eta_1\ \eta_2\ \dots \ \eta_\nu\ \eta_{\nu+1}]$ and $F_{22}$ free parameters.\\
Finally, let $\bar\Pi=[\bar\Pi_0\ \widetilde\Pi]$, where $\widetilde\Pi=[\bar\Pi_1\ \dots\ \bar\Pi_\nu]\in\mathbb{R}^{n\times \nu}$. In this case, (\ref{eq_mom_Markov_bar}) becomes
$$[0_{\nu}\ A\widetilde\Pi]\left[\begin{array}{cc} 0 & S_1 \\ 0_{\nu} & S_2\end{array}\right]+[B\ 0_{n\times\nu}]\left[\begin{array}{cc} 0 & S_1 \\ 0_{\nu} & S_2\end{array}\right]=[\bar\Pi_0\ \widetilde\Pi],$$ yielding the equations
\begin{equation}\label{eq_mom_Markov_tilde}
\bar \Pi_0=0,\ A\widetilde\Pi S_2+BS_1=\widetilde\Pi.
\end{equation} Note that the first Markov parameter is $\eta_0(\infty)=C\bar\Pi_0=0$ and the remaining $\nu$ Markov parameters are in a one-to-one relation with $C\widetilde\Pi$. A reduced order model of dimension $\nu$ that matches the first $\nu+1$ Markov parameters, i.e. $\eta_0(\infty)=0$ and $\eta_1(\infty),\ \dots,\ \eta_{\nu}(\infty)$, is given by system (\ref{red_mod_FGH}) if there exists an invertible matrix $\widetilde P\in\mathbb{R}^{\nu\times\nu}$ such that
\begin{equation}\label{cond_FGH_tilde}
C\widetilde\Pi=H\widetilde P,\ F\widetilde PS_2+GS_1=\widetilde P.
\end{equation}
\vskip -0.3cm\fin\end{exmp}

\subsection{Alternative approach}\label{sect_Markov_dual}

In this subsection we give an alternative interpretation of the notion of moment matching for $\widetilde K$, in the sense of the approach taken in \cite{astolfi-CDC2010}. The moments of $\widetilde K(\tau)$ are in one-to-one relation with the well-defined steady-state response of the asymptotically stable descriptor systems (\ref{sys_descript_1}) or (\ref{sys_descript_2}) driving a generalized signal generator. This approach yields new families of reduced order models that achieve moment matching at $\tau$ and in particular, Markov parameter matching when $\tau=0$.

\begin{prop}\label{prop_mom_Markov_Y}
Consider the system (\ref{lin_sys}) and $\tau^*\in\mathbb{C}$. Let
\begin{equation*}\label{Q} {\mathcal Q}=\left[\begin{array}{ccccc}\tau^* &0 &0 & \dots &0 \\ 1& \tau^* &0 &\dots &0 \\ \vdots & \vdots & \ddots & \ddots & \vdots \\ 0 & \dots &1 &\tau^* &0 \\ 0 &\dots &\dots &1 &\tau^* \end{array}\right]\in\mathbb{C}^{(\nu+1)\times(\nu+1)}
\end{equation*}
be such that Assumption \ref{ass_eigenvalues_Markov} holds and
\begin{equation*}\label{R}
{\mathcal R}=[1\ 0\ 0\ \dots\ 0]^*\in\mathbb{R}^{\nu+1}.
\end{equation*}
\begin{enumerate}

\item Let $\Upsilon\in\mathbb{C}^{(\nu+1)\times n}$ be the unique solution of the Sylvester equation
\begin{equation}\label{eq_matrix_Y_Markov}
\Upsilon={\mathcal Q}\Upsilon A+{\mathcal R}C.
\end{equation} Then the moments of $\widetilde K(\tau)=K(1/\tau)$ at $\tau^*$ satisfy
\begin{equation*}\label{eq_mom_Y_Markov}
[\widetilde\eta_0(\tau^*)\ \widetilde\eta_1(\tau^*)\ \dots\ \widetilde\eta_\nu(\tau^*)]^*={\mathcal Q}\Upsilon B.
\end{equation*}

\item Let $\widehat\Upsilon\in\mathbb{C}^{(\nu+1)\times n}$ be the unique solution of the Sylvester equation
\begin{equation}\label{eq_matrix_Y_Markov_hat}
\widehat\Upsilon={\mathcal Q}\widehat\Upsilon A+{\mathcal R}CA.
\end{equation} Then the moments of $\widetilde K(\tau)=K(1/\tau)$ at $\tau^*$ satisfy
\begin{equation*}\label{eq_mom_Y_Markov_hat}
[\widetilde\eta_0(\tau^*)\ \widetilde\eta_1(\tau^*)\ \dots\ \widetilde\eta_\nu(\tau^*)]^*={\mathcal Q}({\mathcal Q}\widehat\Upsilon+{\mathcal R}C)B.
\end{equation*} \fin

\end{enumerate}
\end{prop}
\begin{proof} The assumption that $\lambda\tau^*\ne 1$, for any $\lambda\in\sigma(A)$ and the regularity of the matrix pencil $I-A\tau$ yield that the solutions $\Pi$ and $\bar\Pi$ of the generalized Sylvester equations (\ref{eq_matrix_Y_Markov}) and (\ref{eq_matrix_Y_Markov_hat}), respectively, exist and are unique (see e.g., \cite{chu-LAA1987}).
\\
To prove the statement (1) let $\Upsilon=[\Upsilon_0^*\ \Upsilon_1^*\ \dots\ \Upsilon_\nu^*]^*$. From (\ref{eq_matrix_Y_Markov}), we have the sequence of equalities 
\begin{align*}
\tau^*\Upsilon_0 A+C=\Upsilon_0 & \Leftrightarrow \Upsilon_0=C(I-A\tau^*)^{-1}, 
\\
\Upsilon_0 A+\tau^*\Upsilon_1 A=\Upsilon_1 &\Leftrightarrow \Upsilon_1=\Upsilon_0 A(I-A\tau^*)^{-1} \\ & \hspace{0.96cm} =C(I-A\tau^*)^{-1}A(I-A\tau^*)^{-1} \\& \hspace{0.96cm}  =CA(I-A\tau^*)^{-2}, 
\\
& \hskip 0.22cm\vdots 
\\
\Pi_{\nu-1} A+ \tau^*\Upsilon_\nu A=\Upsilon_\nu &\Leftrightarrow \Upsilon_\nu =\Upsilon_{\nu-1}(I-A\tau^*)^{-1} \\& \hspace{0.96cm} =CA^\nu(I-A\tau^*)^{-\nu-1}, 
\end{align*}   from which the claim follows.

\noindent To prove the second statement let $\widehat\Upsilon=[\widehat\Upsilon_0^*\ \widehat\Upsilon_1^*\ \dots\ \widehat\Upsilon_\nu^*]^*$. From (\ref{eq_matrix_Y_Markov_hat}) we have the sequence of equalities 
\begin{align*}
\tau^*\widehat\Upsilon_0 A+CA=\widehat\Upsilon_0 &\Leftrightarrow \widehat\Upsilon_0=CA(I-A\tau^*)^{-1}, 
\\
\widehat\Upsilon_0 A+\tau^*\widehat\Upsilon_1 A=\widehat\Upsilon_1 &\Leftrightarrow \bar\Upsilon_1=\widehat\Upsilon_0 A(I-A\tau^*)^{-1} \\& \hspace{0.96cm} =\tau^* C A^2(I-A\tau^*)^{-2}, 
\\
& \hskip 0.22cm\vdots 
\\
\bar\Upsilon_{\nu-1} A+\tau^* \bar\Upsilon_\nu A=\bar\Upsilon_\nu &\Leftrightarrow \bar\Upsilon_\nu=\tau^*CA^\nu(I-A\tau^*)^{-\nu}, 
\end{align*}   from which the claim follows.
\end{proof}
\begin{rem}\n
Let $\bar{\mathcal Q}$ and $\bar{\mathcal R}$ be any two matrices such that the pair $(\bar{\mathcal Q},\bar{\mathcal R})$ is controllable. Then the moments described by (\ref{eq_mom_Y_Markov_hat}) are parameterized by the elements of $\bar{\mathcal{R}}$. However, there exists a coordinate transformation $T$ such that $T^{-1}\bar{\mathcal Q}T={\mathcal Q}$, as in (\ref{Q}) and $T^{-1}\bar{\mathcal R}={\mathcal R}$, as in (\ref{R}), yielding the moments of $\widetilde K(\tau)$ at $\sigma(\bar{\mathcal Q})=\sigma({\mathcal Q})$.
\fin\end{rem}
\begin{figure}[h]\centering
\psfrag{Sys}{\small\hspace{0cm} $\begin{array}{ll} A\dot x=x-B u \\ y=C\dot x\end{array}$} \psfrag{Sig}{\hspace{-0.25cm} $\dot \omega={\mathcal Q}\omega+{\mathcal R}v$}
\psfrag{XT}{\hspace{-0.9cm} $d=\omega-\Upsilon x$} \psfrag{Y}{$v=y$} \psfrag{W_T}{$u$}
\includegraphics[scale=0.7]{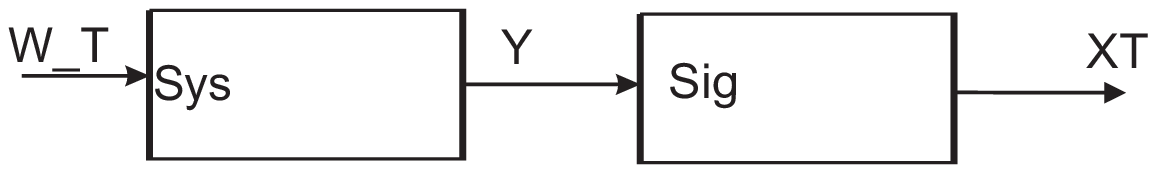}\caption{Interconnection of the signal generator and the system (\ref{sys_descript_2}), with the transfer function $\widetilde K(\tau)$.} \label{diagram_omega_block}
\end{figure}

\noindent Consider the generalized signal generator
\begin{equation}\label{gen_signal_gen}
\dot \omega = {\mathcal Q}\omega + {\mathcal R}v,\ \omega(0)=0,
\end{equation} with $\omega(t)\in\mathbb{R}^\nu$, $v(t)\in\mathbb{R}$.
\noindent Consider the interconnection $v=y$ (as in Fig. \ref{diagram_omega_block}) between the signal generator and the system (\ref{sys_descript_2}). Let $d=\omega-\Upsilon x$. The signal $d(t)$ satisfies $\dot d=\dot\omega-\Upsilon\dot x \Rightarrow \dot d={\mathcal Q}\omega+({\mathcal R}C-\Upsilon)\dot x$. This further yields $\dot d={\mathcal Q}\omega - {\mathcal Q}\Upsilon A\dot x = {\mathcal Q}\omega - {\mathcal Q}\Upsilon (x-Bu)$. In conclusion $d$ satisfies the equation
\begin{equation}\label{eq_signal_d}
\dot d = {\mathcal Q}d+{\mathcal Q}\Upsilon B u.
\end{equation}
\noindent Conversely, let $\Upsilon$ be such that $d=\omega-\Upsilon x$ satisfies equation (\ref{eq_signal_d}). Then $\Upsilon$ satisfies ${\mathcal Q}d+{\mathcal Q}\Upsilon Bu=\dot\omega-\Upsilon \dot x = {\mathcal Q}\omega + {\mathcal R}C\dot x - \Upsilon \dot x$, which yields ${\mathcal Q}\Upsilon x - {\mathcal Q}\Upsilon Bu = ({\mathcal R}C - \Upsilon)\dot x$. Hence ${\mathcal Q}\Upsilon A\dot x=({\mathcal R}C - \Upsilon)\dot x$, for all $x$, i.e., $\Upsilon$ satisfies equation (\ref{eq_matrix_Y_Markov}).
\begin{figure}[h]\centering
\psfrag{Sys}{\small $\begin{array}{ll} A\dot x=x-B\dot u \\ y=C x\end{array}$} \psfrag{Sig}{\hspace{-0.25cm} $\dot \omega={\mathcal Q}\omega+{\mathcal R}v$}
\psfrag{XT}{\scriptsize\hspace{-1cm} $\widehat d={\mathcal Q}(\omega-\widehat\Upsilon x)$} \psfrag{Y}{\hspace{-0.1cm} $v=y$} \psfrag{W_T}{$u$}
\includegraphics[scale=0.7]{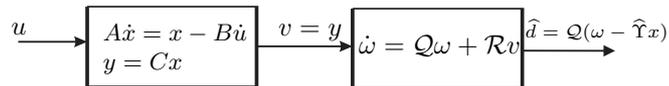}\caption{Interconnection of the signal generator and the system (\ref{sys_descript_1}), with the transfer function $\widetilde K(\tau)$.} \label{diagram_omega_block_1}
\end{figure}
\noindent Consider the generalized signal generator (\ref{gen_signal_gen}) interconnected through $v=y$ (as in Fig. \ref{diagram_omega_block_1}) with the system (\ref{sys_descript_1}).
Let $\widehat d={\mathcal Q}(\omega-\widehat\Upsilon x)$. The signal $\widehat d(t)$ satisfies $\dot{\widehat d}={\mathcal Q}(\dot\omega - \widehat\Upsilon\dot x)={\mathcal Q}({\mathcal Q}\omega +({\mathcal R}CA-\widehat\Upsilon)\dot x+{\mathcal R}CB\dot u)={\mathcal Q}({\mathcal Q}\omega-{\mathcal Q}\widehat\Upsilon A\dot x+{\mathcal R}CB\dot u)={\mathcal Q}({\mathcal Q}\omega - {\mathcal Q}\widehat\Upsilon(x-B\dot u)+{\mathcal R}CB\dot u)$, hence
\begin{equation}\label{eq_signal_d_hat}
\dot{\widehat d}={\mathcal Q}\widehat d + {\mathcal Q}({\mathcal Q}\widehat\Upsilon+{\mathcal R}C)B\dot u.
\end{equation}
\begin{thm}\label{thm_Markov_time_Y}
Consider system (\ref{lin_sys}) and let ${\mathcal Q}$ be such that $\sigma({\mathcal Q})\subset\mathbb{C}^0$. Assume that $\sigma(A)\subset\mathbb{C}^-$ and $x(0)=0$. Let the system (\ref{gen_signal_gen}) be such that the pair $({\mathcal Q},{\mathcal R})$ is controllable.
\begin{description}

\item[{\bf i.}] Consider the interconnection between the system (\ref{sys_descript_1}) and the system (\ref{gen_signal_gen}), defined by $v=y$, with the output $d(t)$. Then the moments $\widetilde\eta_k(\tau_i)$, $k=0,1,2,...$, $i=0,...,\nu$, of the system (\ref{lin_sys}) at $\sigma({\mathcal Q})$, are in one-to-one relation with the well-defined steady-state response of the output $d(t)$, for $u=\delta(t)$.

\item[{\bf ii.}] Consider the interconnection between the system (\ref{sys_descript_2}) and the system (\ref{gen_signal_gen}), defined by $v=y$, with the output $\widehat d(t)$. Then the moments $\widetilde\eta_k(\tau_i)$, $k=0,1,2,...$, $i=0,...,\nu$, of the system (\ref{lin_sys}) at $\sigma({\mathcal Q})$, are in one-to-one relation with the well-defined steady-state response of the output $\widehat d(t)$, for $u=1(t)$, where $1(t)$ denotes the Heaviside step function.
\fin
\end{description}
\end{thm}
\begin{proof} The proof is similar to the proof of Theorem \ref{thm_mom_steady} (see also \cite{astolfi-CDC2010}), hence, we provide a sketch for the statement {\bf i.} By the assumptions, the steady-state response of $d(t)$ is well-defined. By equation (\ref{eq_signal_d}), since $u(t)=\delta(t)$, we have that $D(\tau)=(\tau I-{\mathcal Q})^{-1}{\mathcal Q}\Upsilon B$, where $D(\tau)$ denotes the Laplace transform of $d(t)$. $D(\tau)$ contains the moments in the terms $\widetilde K(\tau^*){\mathcal R}/(\tau-\tau^*)$, where $\tau^*\in\sigma({\mathcal Q})$ and $\widetilde K(\tau)$ is the transfer function of the system (\ref{sys_descript_1}). Recalling that $\dot 1(t)=\frac{d 1(t)}{dt}=\delta(t)$, the proof of the statement {\bf ii.} is identical to the proof of statement {\bf i.}
\end{proof}

\begin{cor}
Consider the system (\ref{lin_sys}) satisfying the assumptions in Theorem \ref{thm_Markov_time_Y}. Let $\sigma({\mathcal Q})=\{0,...,0\}$. Then the first $\nu$ Markov parameters of (\ref{lin_sys}) are in a one-to-one relation with the well-defined steady-state responses of the signals $d(t)$ and $\widehat d(t)$, to inputs $u(t)=\delta(t)$ and $u(t)=1(t)$, respectively.
\fin\end{cor}
We are now ready to present new families of reduced order models that achieve moment matching at $\sigma({\mathcal Q})$.
\\
Consider a system described by equations of the form
\begin{equation}\label{red_FGH_left}\Sigma_{\mathbf\Upsilon}:\left\{\begin{split}
& E\dot\xi=F\xi+Gu, \\& \psi=H\xi,
\end{split}\right.\end{equation} with $\xi(t)\in\mathbb{R}^\nu$, $u(t)\in\mathbb{R}$, $\psi(t)\in\mathbb{R}$, and $E,\ F,\ G,\ H$ of appropriate dimensions. Let $K_{\mathbf\Upsilon}=H(sE-F)^{-1}G$ be the well-defined transfer function of (\ref{red_FGH_left}) and let $\widetilde K_{\mathbf\Upsilon}(\tau)=K_{\mathbf\Upsilon}(1/\tau)$.

Let the signal $\zeta(t)$ be such that $\zeta=\omega-PE\xi$. Then, the moments of the transfer function $\widetilde K_{\mathbf\Upsilon}(\tau)$ of (\ref{red_FGH_left}) match the moments of the transfer function $\widetilde K(\tau)$ of (\ref{lin_sys}) at $\sigma({\mathcal Q})$, in the sense of Theorem \ref{thm_Markov_time_Y}, if $\zeta(t)$ satisfies $\dot\zeta={\mathcal Q}\zeta+{\mathcal Q}\Upsilon Bu$, which is equivalent to the existence of an invertible matrix $P$ such that
\begin{equation}\label{markov_match_cond_1}
{\mathcal Q}PF+{\mathcal R}H=PE,\ {\mathcal Q}PG={\mathcal Q}\Upsilon B.
\end{equation}
Similarly, let $\widehat\zeta(t)$ be such that $\widehat\zeta={\mathcal Q}(\omega-\widehat PE\xi)$. Then, the moments of the transfer function $\widetilde K_{\mathbf\Upsilon}(\tau)$ of (\ref{red_FGH_left}) match the moments of the transfer function $\widetilde K(\tau)$ of (\ref{lin_sys}) at $\sigma({\mathcal Q})$, in the sense of Theorem \ref{thm_Markov_time_Y}, if $\widehat \zeta(t)$ satisfies $\dot{\widehat\zeta}={\mathcal Q}\widehat\zeta+{\mathcal Q}({\mathcal Q}\widehat\Upsilon +{\mathcal R}C)Bu$, which is equivalent to the existence of an invertible matrix $P$ such that
\begin{equation}\label{markov_match_cond_2}
{\mathcal Q}\widehat PF+{\mathcal R}H=\widehat PE,\ {\mathcal Q}({\mathcal Q}\widehat P+{\mathcal R}H)G={\mathcal Q}({\mathcal Q}\widehat \Upsilon +{\mathcal R}C)B.
\end{equation}
In the sequel, we present particular instances of $E$, $F$, $G$ and $H$, which satisfy the matching relations (\ref{markov_match_cond_1}) and (\ref{markov_match_cond_2}), respectively.
\begin{prop}\label{prop_redmod_F_id}
Consider system (\ref{lin_sys}) with the transfer function $K(s)$. Let $\widetilde K(\tau)=K(1/\tau)$ and let $\Upsilon$ and $\widehat \Upsilon $ be the unique solutions of the Sylvester equations (\ref{eq_matrix_Y_Markov}) and (\ref{eq_mom_Y_Markov_hat}), respectively. Furthermore, let ${\mathcal Q}$ and ${\mathcal R}$ be such that the pair $({\mathcal Q},{\mathcal R})$ is controllable. Then the following statements hold.
\begin{enumerate}
\item A reduced order model with transfer function $\widetilde K_{\mathbf\Upsilon}(\tau)$ which matches the moments of $\widetilde K(\tau)$ at $\sigma({\mathcal Q})$ is given by the system (\ref{red_FGH_left}) with $E={\mathcal Q}-{\mathcal R}H$, $F=I$ and $G=-\Upsilon B$.
\item A reduced order model with transfer function $\widetilde K_{\mathbf\Upsilon}(\tau)$ which matches the moments of $\widetilde K(\tau)$ at $\sigma({\mathcal Q})$ is given by the system (\ref{red_FGH_left}) with $E={\mathcal Q}-{\mathcal R}H$, $F=I$ and $G=-{\mathcal Q}\Upsilon B$ and output $\psi=H\dot\xi$.
\item A reduced order model with transfer function $\widetilde K_{\mathbf\Upsilon}(\tau)$ which matches the moments of $\widetilde K(\tau)$ at $\sigma({\mathcal Q})$ is given by the system (\ref{red_FGH_left}) with $E={\mathcal Q}-{\mathcal R}H$, $F=I$ and $G=-({\mathcal Q}\widehat\Upsilon +{\mathcal R}C)B$.
\item A reduced order model with transfer function $\widetilde K_{\mathbf\Upsilon}(\tau)$ which matches the moments of $\widetilde K(\tau)$ at $\sigma({\mathcal Q})$ is given by the system (\ref{red_FGH_left}) with $E={\mathcal Q}-{\mathcal R}H$, $F=I$ and $G=-{\mathcal Q}({\mathcal Q}\widehat\Upsilon +{\mathcal R}C)B$ and input $\dot u(t)$. \fin
\end{enumerate}
\end{prop}
\begin{proof} We only prove statement (1), since the statements (2)-(4) follow similar arguments.
Let $K(\tau)=C(I-A\tau)^{-1}B\tau$ and let $\widetilde K_{\mathbf\Upsilon}(\tau)=H(\tau I-{\mathcal Q}+{\mathcal R}H)^{-1}\Upsilon B\tau$. Hence $E(\tau)=K(\tau)-\widetilde K_{\mathbf\Upsilon}(\tau)=[C-H(\tau I-{\mathcal Q}+{\mathcal R}H)^{-1}\Upsilon(I-A\tau)](I-A\tau)^{-1}B\tau$. Exploiting equation (\ref{eq_mom_Y_Markov}) and performing some algebraic computations yield $E(\tau)=\widetilde E(\tau)(C+H\Upsilon A)(I-A\tau)^{-1}B\tau$, where $\widetilde E: \mathbb{C}\to\mathbb{C}$, $\widetilde E(\tau)=1-H(\tau I-{\mathcal Q}+{\mathcal R}H)^{-1}{\mathcal R}$. Note that
$$\widetilde E(\tau)=1-\frac{H(\tau I-{\mathcal Q})^{-1}{\mathcal R}}{1+H(\tau I-{\mathcal Q})^{-1}{\mathcal R}}=\frac{1}{1+H(\tau I-{\mathcal Q})^{-1}{\mathcal R}}.$$
It follows that for all $\tau_i\in\sigma({\mathcal Q}),\ i=1,...,\nu$, $\widetilde E(\tau_i)=0$ and then $E(\tau_i)=0$, which proves the claim. The rest of the claims are proven in a similar way.
\end{proof}
\begin{rem}\n
Let $F=I$. Then a reduced order model that achieves moment matching of $\widetilde K(\tau)$ at $\sigma({\mathcal Q})$ as in Proposition \ref{prop_redmod_F_id}, is obtained for $P=-I$ or $\widehat P=-I$, where $P$ uniquely satisfies (\ref{markov_match_cond_1}) and $\widehat P $ uniquely satisfies (\ref{markov_match_cond_2}), respectively. Note that the uniqueness of the solutions holds if and only if $\sigma({\mathcal Q})\cap\sigma({\mathcal Q}-{\mathcal R}H)=\emptyset$.
\fin\end{rem}
\begin{exmp}\n
Let
$${\mathcal Q}=\left[\begin{array}{ccc} 0 & 0 & 0 \\ 1 & 0 & 0 \\ 0 & 1 & 0\end{array}\right],\ {\mathcal R}=\left[\begin{array}{c} 1 \\ 0 \\ 0\end{array}\right]$$
and consider a reduced order model (\ref{red_FGH_left}) with $E=I$, $G=[g_1\ g_2\ g_3]^*$ and $H=[h_1\ h_2\ h_3]$. Let $\Upsilon$ be the unique solution of (\ref{eq_mom_Y_Markov}) and let $\Upsilon B=[\eta_0\ \eta_1\ \eta_2]^*$. System (\ref{red_FGH_left}) matches the first three Markov parameters of the system (\ref{lin_sys}), i.e., the moments ${\mathcal Q}\Upsilon B=[0\ \eta_0\ \eta_1]^*$, for any $P$ that uniquely satisfies (\ref{markov_match_cond_1}). Choosing $P=I$, and solving (\ref{markov_match_cond_1}) yields
$$F=\left[\begin{array}{ccc} 0 & 1 & 0 \\ 0 & 0 & 1 \\ -a & -b & -c \end{array}\right],\ G=\left[\begin{array}{c} \eta_0 \\ \eta_1 \\ g_3\end{array}\right],\ H=[1\ 0\ 0],$$
with the transfer function 

$$ \widehat K(s)=\frac{\eta_0 s^2+(\eta_1-\eta_0 c)s-\eta_0 b-\eta_1 c+g_3}{s^3+s^2c+sb+a},$$
where $a,b,c$ and $g_3\in\mathbb{R}$ are free parameters.
In general, let
$${\mathcal Q}=\left[\begin{array}{cc} 0 & 0 \\ I_{(\nu-1)\times (\nu-1)} & 0\end{array}\right],\
{\mathcal R}=\begin{tikzpicture}[decoration={brace,amplitude=5pt},baseline=(current bounding box.west)]
    \matrix (magic) [matrix of math nodes,left delimiter=[,right delimiter={]} ] {
      1 \\
      0 \\
      \vdots  \\
      0 \\
    };
    \draw[decorate,black,transform canvas={xshift=0.6cm}] ($(magic-2-1.north)$) -- (magic-4-1.south) node[right=2pt, midway] {{\scriptsize $\nu-1$}};
  \end{tikzpicture}$$
and $E=I$. Choosing $P=I$ and solving (\ref{markov_match_cond_1}) yield a class of reduced order models that match $\nu$ Markov parameters of the system (\ref{lin_sys}) described by
$$F=\left[\begin{array}{cc} 0 & I_{\nu-1} \\ f_1 & F_2^*\end{array}\right],\ G=\left[\begin{array}{c} \Xi \\ g\end{array}\right],\ H=[1\ \underbrace{0\ \dots\ 0}_{\nu-1}],$$
with $F_2\in\mathbb{R}^{\nu-1}$ and $f_1\in\mathbb{R}$ and $g\in\mathbb{R}$ free parameters, and $\Xi$ such that ${\mathcal Q}\Upsilon B=[0\ \ \Xi^*]^*$. Note that $\lambda\mu=0\ne 1$ for all $\lambda\in\sigma(F)$, since $0=\mu\in\sigma({\mathcal Q})$, and thus $P=I$ uniquely satisfies (\ref{markov_match_cond_1}).
\fin\end{exmp}

\begin{exmp}\n
Let ${\mathcal Q}\in\mathbb{R}^{\nu\times\nu}$ and ${\mathcal R}\in\mathbb{R}^\nu$ be such that the pair $({\mathcal Q},{\mathcal R})$ is controllable and let $\widehat\Upsilon$ be the unique solution of equation (\ref{eq_mom_Y_Markov}). Furthermore, let $\widehat P=I$ be the unique solution of (\ref{markov_match_cond_2}) and assume $E=I$. Then a class of reduced order models that match the moments of $K(\tau)$ at $\sigma({\mathcal Q})$ is given by
$$F=({\mathcal Q}+{\mathcal R}H)^{-1},\ G=({\mathcal Q}+{\mathcal R}H)^{-1}({\mathcal Q}\widehat\Upsilon +{\mathcal R}C)B,$$
with $H$ a free parameter, such that ${\mathcal Q}+{\mathcal R}H$ is invertible and $\lambda\mu\ne 1$, for any $\lambda\in\sigma(F)$ and $\mu\in\sigma({\mathcal Q})$.

\noindent In particular, let ${\mathcal Q}$ be a Jordan block of order $\nu$ with all eigenvalues at zero, ${\mathcal R}=[1\ 0\ ...\ 0]^*\in\mathbb{R}^\nu$, $({\mathcal Q}\widehat\Upsilon +{\mathcal R}C)B=[\eta_0\ \eta_1\ ...\ \eta_{\nu-1}]^*$ and $H=[h_1\ h_2\ ...\ h_\nu]\in\mathbb{R}^{1\times\nu}$. Then $\det({\mathcal Q}+{\mathcal R}H)=h_\nu$. Choosing $h_\nu\ne 0$, we obtain a class of reduced order models described by equations (\ref{red_FGH_left}) and parameterized by $h_1,...,h_\nu$, that matches the first $\nu$ Markov parameters of the system (\ref{lin_sys}), with
$$F=\left[\begin{array}{ccccc} 0 & 1 & 0 & \dots & 0 \\ 0 & 0 & 1 & \dots & 0 \\ \vdots & \vdots & \vdots &\ddots & \vdots\\ 0 & 0 & 0 & \dots &1 \\ \frac{1}{h_\nu} & -\frac{h_1}{h_\nu} & -\frac{h_2}{h_\nu} & \dots & -\frac{h_{\nu-1}}{h_\nu}\end{array}\right],\ G=\left[\begin{array}{c} \eta_1 \\ \vdots \\ \eta_{\nu-1} \\ g\end{array}\right],$$
and $\displaystyle g=\frac{\eta_0-\sum_{i=1}^{\nu-1}h_i\eta_i}{h_\nu}$. Note that by construction $0\notin \sigma(F)$ and so $\lambda\mu=0$ for all $\lambda\in\sigma(F)$, since $0=\mu\in\sigma({\mathcal Q})$, i.e., $\widehat P=I$ uniquely satisfies (\ref{markov_match_cond_2}).
\fin\end{exmp}

\subsection{Relation with Krylov methods}

In this section we establish a connection between the Krylov projection in the Arnoldi method (see, e.g., \cite{antoulas-2005,polyuga-vdschaft-AUT2010}) and the solution of the Sylvester equation (\ref{eq_matrix_Pi_Markov}). We will only sketch the result which follows arguments from the proof of Lemma \ref{lemma_V_and_Pi}. Consider the system (\ref{lin_sys}) and let $V\in\mathbb{C}^{n\times \nu}$ and $W\in\mathbb{C}^{n\times \nu}$ be such that
\begin{subequations}\label{eq_WV_Markov}
\begin{eqnarray}
V &=& [B\ AB\ ...\ A^{\nu-1}B], \label{eq_V_Markov} \\
W &=& [C^*\ (CA)^*\ ...\ (CA^{\nu-1})^*]. \label{eq_W_Markov}
\end{eqnarray}
\end{subequations}
Let $S$ be as in Proposition \ref{prop_mom_Markov} and note that $AVS+BL=V$, with $L$ such that the pair $(L,S)$ is observable. Let $T$ be an invertible matrix such that $S=T\bar ST^{-1}$, where $\bar S\in\mathbb{C}^{\nu\times\nu}$ is any matrix such that $\sigma(S)=\sigma(\bar S)$. Hence $AVT\bar S+BLT=VT$ yielding that the unique solution $\Pi$ of (\ref{eq_matrix_Pi_Markov}) satisfies $\Pi=VT$. Similarly, $\bar\Pi=VT$, where $\bar\Pi$ is the unique solution of (\ref{eq_matrix_Pi_Markov_bar}). Note that similar results are obtained for $\Upsilon$, the unique solution of (\ref{eq_matrix_Y_Markov}) or (\ref{eq_matrix_Y_Markov_hat}), i.e., $\Upsilon=TW^*$, where $T$ is an invertible matrix such that ${\mathcal Q}=T\bar{\mathcal Q}T^{-1}$, with ${\mathcal Q}\in\mathbb{C}^{\nu\times \nu}$ as in Proposition {\ref{prop_mom_Markov_Y}}.

The following result establishes a relation between the class of $\nu$-th order models that match a set of prescribed Markov parameters, obtained using Krylov projections and the class of reduced models $\Sigma_{\mathbf\Pi}$. A similar result is obtained for a class of models $\Sigma_{\mathbf\Upsilon}$.
\begin{prop}\label{prop_equiv_Markov}
Consider the system (\ref{lin_sys}) and the system
\begin{equation}\label{sys_WAV_Markov}\begin{split}
& \dot \xi=W^*AV\xi+W^*Bu, \\& \psi=CV\xi,
\end{split}\end{equation}
with $\xi(t)\in\mathbb{R}^\nu$. The following statements hold.
\begin{enumerate}
\item Let $V$ be as in (\ref{eq_V_Markov}) and let $W$ be such that $W^*V$=I. Let $\Sigma_W$ be the family of $\nu$ order models of (\ref{lin_sys}), described by equations (\ref{sys_WAV_Markov}) and parameterized in $W$. Then, $\Sigma_W=\Sigma_{\mathbf\Pi}$, with $\Sigma_{\mathbf\Pi}$ as in (\ref{red_mod_FGH}), with $\mathbf\Pi\in\{\Pi,\bar\Pi\}$, where $\Pi$ is the unique solution of (\ref{eq_matrix_Pi_Markov}) and $\bar\Pi$ is the unique solution of (\ref{eq_matrix_Pi_Markov_bar}). Analogously, let $W$ be as in (\ref{eq_W_Markov}) and let $V$ be such that $W^*V$=I. Let $\Sigma_V$ be the family of $\nu$ order models of (\ref{lin_sys}), described by equations (\ref{sys_WAV_Markov}) and parameterized in $V$. Then, $\Sigma_V=\Sigma_{\mathbf\Pi}$.

\item Let $W$ be as in (\ref{eq_W_Markov}) and let $V$ be such that $W^*V$=I. Let $\Sigma_V$ be the family of $\nu$ order models of (\ref{lin_sys}), described by equations (\ref{sys_WAV_Markov}) and parameterized in $V$. Then, $\Sigma_V=\Sigma_{\Upsilon}$, with $\Sigma_{\Upsilon}$ as in (\ref{red_FGH_left}), with $\Upsilon$ the unique solution of (\ref{eq_matrix_Y_Markov}). Analogously, let $V$ be as in (\ref{eq_V_Markov}) and let $W$ be such that $W^*V$=I. Let $\Sigma_W$ be the family of $\nu$ order models of (\ref{lin_sys}), described by equations (\ref{sys_WAV_Markov}) and parameterized in $W$. Then, $\Sigma_W=\Sigma_{\Upsilon}$. \fin
\end{enumerate}
\end{prop}
\begin{proof}
{\it Proof of Statement (1)}. Consider a model $\Sigma_W$ and let $\Pi=VP$, with $\Pi$ the unique solution of (\ref{eq_matrix_Pi_Markov}) and $P$ an invertible matrix. Thus, $CVP=C\Pi$ which yields $C\Pi S=CVPS$. Furthermore, note that $W^*AVPS+W^*BL=W^*(A\Pi S+BL)=W^*VP=P$ for any $W$ such that $W^*V=I$. Hence, the conditions (\ref{cond_FGH})  are satisfied and the claim is proven. Similar arguments hold for $\mathbf\Pi=\bar\Pi$, the unique solution of (\ref{eq_matrix_Pi_Markov_bar}), as well as for the case $\Sigma_V$.

{\it Proof of Statement (2)}. Let $\Upsilon$ be the unique solution of (\ref{eq_matrix_Y_Markov}) and consider the family of $\nu$-th order models $\Sigma_\Upsilon$, as in (\ref{red_FGH_left}). Without loss of generality assume $E=I$ in (\ref{red_FGH_left}). Consider a model $\Sigma_V$ and let $\Upsilon = PW^*$ with $P$ an invertible matrix. Hence, ${\mathcal Q}PW^*B=Q\Upsilon B$. Furthermore, note that ${\mathcal Q}PW^*AV+{\mathcal R}CV=({\mathcal Q}\Upsilon A+{\mathcal R}C)V=\Upsilon V=PW^*V=P$, for any $V$ such that $W^*V=I$. Thus, the conditions (\ref{markov_match_cond_1}) are satisfied and the claim is proven. Similar arguments hold for the case $\Sigma_W$.
\end{proof}

Note that the results from Proposition \ref{prop_equiv_Markov} do not apply to the family of models $\Sigma_{\widehat\Upsilon}$, with $\widehat\Upsilon$ the unique solution of (\ref{eq_matrix_Y_Markov_hat}).

\section{Matching with preservation of the port Hamiltonian structure}\label{sect_pH_match}

\subsection{Matching at finite interpolation points}

 In this section, we solve the following particular instance of the \ref{prob_redmod_prose}, at a prescribed set of finite interpolation points.

\begin{prob}\label{prob_redmod}\n
Given a linear, port Hamiltonian, SISO system (\ref{pH_system}), find the observable pair $(L,S)$, where $L\in\mathbb{C}^{1\times \nu}$ and $S\in\mathbb{C}^{\nu\times\nu}$, such that $\sigma(S)\cap\sigma((J-R)Q)=\emptyset$ and the free parameter $G\in\mathbb{C}^\nu$, such that the following properties hold.
\begin{description}
\item[(p1)] The family of systems $(S-GL,G,B^*Q\Pi)$ parameterized in $S,\ L,\ G$, described by equations (\ref{model_gen_G}), with $\Pi$ the unique solution of (\ref{eq_Sylvester}), match the moments of (\ref{pH_system}) at $\sigma(S)$.

\item[(p2)] There exists $G$ such that the family of systems $(S-GL,G,B^*Q\Pi)$ parameterized in $S$ and $L$, described by equations (\ref{model_gen_G}) preserve the port Hamiltonian structure.

\item[(p3)] The family of systems $(S-GL,G,B^*Q\Pi)$ parameterized in $S,\ G,\ L$, described by equations (\ref{model_gen_G}) are accurate approximations of (\ref{pH_system}), in the $H_2$-norm sense.

\item[(p4)] The computation of the family of systems $(S-GL,G,B^*Q\Pi)$ parameterized in $S,\ G,\ L$, described by equations (\ref{model_gen_G}) is calculated efficiently, i.e., avoiding the computation of the moments and the solution $\Pi$ of (\ref{eq_Sylvester}). \fin
\end{description}
\end{prob}
The solution to subproblem {\bf (p1)} is provided by Theorem \ref{thm_redmod_CPi} which characterizes the family of systems $(S-GL,G,B^*Q\Pi)$ parameterized in $S,\ L,\ G$. Our main goal is to solve subproblem {\bf (p2)}. In short, fixing $S$ and $L$ we compute the $\nu$-th order system $(S-GL,G,B^*Q\Pi)$ that preserves the port Hamiltonian structure and matches the moments of (\ref{pH_system}) at $\sigma(S)$. This model characterizes a class of state-space models parameterized in $S$ and $L$, which have the same transfer function. The parameters $S$ and $L$ can be selected to satisfy further constraints such as additional physical structure and improved accuracy of the approximations, by solving subproblem {\bf (p2)} at the mirror images of a set of poles of the given system (\ref{pH_system}), i.e., the port Hamiltonian reduced order model satisfies conditions (\ref{eq_H2opt1}), thus shortly addressing subproblem {\bf (p3)}. In the case of MIMO systems, a proper selection of $L$ defines the desired directions for the solution of the tangential interpolation problem, i.e., the reduced order port Hamiltonian system satisfies the right/left tangential interpolation conditions (\ref{eq_tangent}) or (\ref{eq_tangent_dual}), respectively. Although it is not the primary goal of this paper, subproblem {\bf (p4)} is addressed using the result in Theorem \ref{thm_red_PH_WV} and the relation between the projections and the solutions of the Sylvester equations (\ref{eq_Sylvester}) and (\ref{eq_Sylvester_Y}), provided by Lemma \ref{lemma_V_and_Pi}. Hence, the output $B^*Q\Pi$ of (\ref{model_gen_G}) is computed using efficient numerical algorithms stemming from Krylov subspace techniques as in, e.g., \cite{gugercin-polyuga-beattie-vdschaft-ARXIV2011}.

Note that a similar problem and results are obtained using the family of  systems $({\mathcal Q}-{\mathcal R}H,\Upsilon B,H)$ parameterized in ${\mathcal Q},{\mathcal R},\ H$, as in (\ref{model_gen_H}), with $\Upsilon$ the unique solution of (\ref{eq_Sylvester_Y}).

Our main focus is solving subproblem {\bf (p2)} of Problem \ref{prob_redmod}. The results consist of families of reduced order, port Hamiltonian models which are subclasses of the class of models $\Sigma_G$ in (\ref{model_gen_G}) and $\Sigma_H$ in (\ref{model_gen_H}), respectively. The port Hamiltonian models match the moments of the given port Hamiltonian system and possess parameterized state-space realizations. Note that all the model from the subclass have the same transfer function. The parameters can be used for enforcing additional structure such as, e.g., specific Hamiltonian functions and/or diagonalized dissipation. In the MIMO case, the parameters can be used to find the reduced order models that satisfy the tangential interpolation conditions.

Consider the linear, SISO, port Hamiltonian system (\ref{pH_system}). Let $(L,S)$ be an observable pair and let $({\mathcal Q},{\mathcal R})$ be a controllable pair. Consider the families of reduced order models $\Sigma_G$ as in (\ref{model_gen_G}) and $\Sigma_H$ as in (\ref{model_gen_H}), respectively. Throughout the rest of this section we make the following working assumption.
\begin{assum}\label{ass_Q_S_Q}\n
The matrix $Q$ is invertible. Furthermore, $\sigma(S)\cap\sigma((J-R)Q)=\emptyset$, $\sigma({\mathcal Q})\cap\sigma(A)=\emptyset$, $\sigma(S)\cap\sigma(S-GL)=\emptyset$ and $\sigma({\mathcal Q})\cap\sigma({\mathcal Q}-{\mathcal R}H)=\emptyset$, i.e., the interpolation points are not among the poles of either the given system or its approximations.
\end{assum}
Note that by Assumption \ref{ass_Q_S_Q}, the families of models $\Sigma_G$ and $\Sigma_H$ are well-defined.
\begin{prop}\label{prop_red_pH_Pi}
Consider the port Hamiltonian reduced order model given by
\begin{equation}\label{model_pH_Pi} \begin{split}
& \dot \xi= (\widetilde J-\widetilde R)\widetilde Q\xi+\widetilde Bu,\\& \psi=\widetilde B^*\widetilde Q\xi,
\end{split}\end{equation} with $\xi(t)\in\mathbb{R}^{\nu}$. Then, the following statements hold.
\begin{enumerate}
\item Let $\Pi$ be the unique solution of equation (\ref{eq_Sylvester}) and define
\begin{equation}\label{model_pH_Pi_paramters}\begin{split}
&\widetilde J=\Pi^*QJQ\Pi,\ \widetilde R=\Pi^*QRQ\Pi, \\& \widetilde Q=(\Pi^*Q\Pi)^{-1},\ \widetilde B=\Pi^*QB.
\end{split}\end{equation} Let $\Sigma_{\Pi}$ be a port Hamiltonian model described by equations (\ref{model_pH_Pi}) and (\ref{model_pH_Pi_paramters}). If $\sigma(S)\cap\sigma((\widetilde J-\widetilde R)\widetilde Q)=\emptyset$, then $\Sigma_\Pi$ matches the moments of the system (\ref{pH_system}) at $\sigma(S)$.

\item Let $\Upsilon$ be the unique solution of equation (\ref{eq_Sylvester_Y}) and define
\begin{equation}\label{model_pH_Y_paramters}\begin{split}
&\widetilde J=\Upsilon J\Upsilon^*,\ \widetilde R=\Upsilon R\Upsilon^*,\\& \widetilde Q=(\Upsilon Q^{-1}\Upsilon^*)^{-1},\ \widetilde B=\Upsilon B.
\end{split}\end{equation}
Let $\Sigma_\Upsilon$ be a port Hamiltonian model described by equations (\ref{model_pH_Pi}) and (\ref{model_pH_Y_paramters}). If $\sigma({\mathcal Q})\cap\sigma((\widetilde J-\widetilde R)\widetilde Q)=\emptyset$, then $\Sigma_{\Upsilon}$ matches the moments of the system (\ref{pH_system}) at $\sigma({\mathcal Q})$.
\fin
\end{enumerate}
\end{prop}
\begin{proof}
Let $K(s)$ be the transfer function of system (\ref{pH_system}).
\\
{\it Proof of Statement (1)}. Let $K_\Pi(s)$ the transfer function of the system $\Sigma_\Pi$. Assuming that $\sigma(S)\cap\sigma((\widetilde J-\widetilde R)\widetilde Q)=\emptyset$, the result is obtained by simply checking the moments at $s_i\in\sigma(S),\ i=1,...,\nu$, i.e., checking that $K(s_i)= K_\Pi(s_i)$. To this end, note that $K_\Pi(s_i)=\widetilde B^*\widetilde Q(s_i I-(\widetilde J-\widetilde R)\widetilde Q)^{-1}\widetilde B=B^*Q\Pi(\Pi^*Q\Pi)^{-1}(s_iI-\Pi^*Q(J-R)Q\Pi(\Pi^*Q\Pi)^{-1})^{-1}\Pi^*QB$. Since $\Pi$ is the solution of (\ref{eq_Sylvester}), $(J-R)Q\Pi=\Pi S-BL$ and hence $K_\Pi(s_i)=B^*Q\Pi(s_iI-S+(\Pi^*Q\Pi)^{-1}QBL)^{-1}(\Pi^*Q\Pi)^{-1}QBL$. By  Proposition \ref{prop_matching_happens}, for $\Delta=(\Pi^*Q\Pi)^{-1}QB$, we have that $K(s_i)=K_\Pi(s_i)$, which yields the result.

{\it Proof of Statement (2)}. Let $K_\Upsilon(s)$ be the transfer function of the system $\Sigma_\Upsilon$. Note that $K_\Upsilon(s_i)=B^*\Upsilon^*(\Upsilon Q^{-1}\Upsilon^*)^{-1}(s_iI-\Upsilon (J-R)\Upsilon^*(\Upsilon Q^{-1}\Upsilon^*)^{-1})^{-1}\Upsilon B$. Since $\Upsilon$ is the solution of (\ref{eq_Sylvester_Y}), we have $\Upsilon (J-R)Q={\mathcal Q}\Upsilon-{\mathcal R}B^*Q$ and hence $\widehat K(s_i)=B^*\Upsilon^*(\Upsilon Q^{-1}\Upsilon^*)^{-1}(s_iI-{\mathcal Q}+{\mathcal R}B^*\Upsilon^*(\Upsilon Q^{-1}\Upsilon^*)^{-1})^{-1}\Upsilon B$. By Proposition \ref{prop_matching_happens}, for $H=B^*\Upsilon^*(\Upsilon Q^{-1}\Upsilon^*)^{-1}$, we have that $K(s_i)=K_\Upsilon(s_i)$, which proves the result.
\end{proof}
\begin{rem}\n
Let (\ref{model_pH_Pi}) be a reduced order model of (\ref{pH_system}). Then, by Theorem \ref{thm_mom_steady}, the model (\ref{model_pH_Pi}) matches the moments of (\ref{pH_system}) at $\sigma(S)$ if
\begin{equation}\label{eq_Sylvester_P}
B^*Q\Pi=\widetilde B^*\widetilde Q P,
\end{equation}  where $P\in\mathbb{R}^{\nu\times\nu}$ is an invertible matrix such that
\begin{equation}\label{eq_Sylvester_P1}
(\widetilde J-\widetilde R)\widetilde Q P+\widetilde B L= P S.
\end{equation} Note that equation (\ref{eq_Sylvester_P}) is satisfied by
\begin{equation}\label{eq_P} P=\widetilde Q^{-1}=\Pi^*Q\Pi. \end{equation}
Since we assume that $\sigma(S)\cap\sigma((\widetilde J-\widetilde R)\widetilde Q)=\emptyset$, then $P$ as in (\ref{eq_P}) is the unique solution of the Sylvester equation (\ref{eq_Sylvester_P1}) (see also \cite{astolfi-TAC2010}).
\fin\end{rem}
\begin{rem}\n
Let systems $\Sigma_\Pi$ as in (\ref{model_pH_Pi}) and $\Sigma_V$ as in (\ref{model_pH_WV}) be two reduced order port Hamiltonian models that match the moments of (\ref{pH_system}), respectively. Then, by Lemma \ref{lemma_V_and_Pi}, they are equivalent\footnote{Two minimal systems described by state-space equations are called equivalent if they have the same transfer functions.}, i.e., there exists an invertible matrix $T$ such that $\Pi T=V$ and $J_r=T^*\widetilde J T$, $R_r=T^*\widetilde R T$, $Q_r=T^{-*}\widetilde Q T^{-1}$ and $B_r=T^*\widetilde B$, hence $\Sigma_V=\Sigma_{\Pi T}$ and furthermore, $\Sigma_V$ is a member of the class of models $\Sigma_G$.
\fin\end{rem}
Note that the result in Proposition \ref{prop_red_pH_Pi} shows that there exists a port Hamiltonian model that matches the moments of a given port Hamiltonian system by direct computation of the matrix $\Pi$. However this can be avoided by showing that the model $\Sigma_{\Pi}$ is a member of the class of reduced order models $\Sigma_G$ that match the moments of (\ref{pH_system}) at $\sigma(S)$, for a particular instance of the parameter $G$. Similarly, we also show that the model $\Sigma_\Upsilon$ is a member of the class of reduced order models $\Sigma_H$ that match the moments of (\ref{pH_system}) at $\sigma(\mathcal{Q})$, for a particular selection of $H$.
\begin{thm}\label{prop_pH_inc_G} Let Assumption \ref{ass_Q_S_Q} hold. Then the following statements hold.
\begin{enumerate}
\item Let $\Sigma_G$, as in (\ref{model_gen_G}), be a reduced order model of the port Hamiltonian system (\ref{pH_system}), matching the moments at $\sigma(S)$. Then $\Sigma_G$ is equivalent to a port Hamiltonian system $\Sigma_\Pi$, as in (\ref{model_pH_Pi}) and (\ref{model_pH_Pi_paramters}), i.e., $\Sigma_G$ preserves the port Hamiltonian structure of the system (\ref{pH_system}), if and only if $G=(\Pi^*Q\Pi)^{-1}\Pi^*QB$.

\item Let $\Sigma_H$, as in (\ref{model_gen_H}), be a reduced order model of the port Hamiltonian system (\ref{pH_system}), matching the moments at $\sigma({\mathcal Q})$. Then $\Sigma_H$ is equivalent to the port Hamiltonian system $\Sigma_\Upsilon$, as in (\ref{model_pH_Pi}) and (\ref{model_pH_Y_paramters}), i.e., $\Sigma_G$ preserves the port Hamiltonian structure of the system (\ref{pH_system}), if and only if $H=B^*\Upsilon^*(\Upsilon Q^{-1}\Upsilon^*)^{-1}$.\fin
\end{enumerate}
\end{thm}
\begin{proof} {\it Proof of Statement (1)}. First we prove the necessity.
Let $\Pi$ satisfy (\ref{eq_Sylvester}). Since $-J=J^*$ and $R=R^*$, $\Pi^*$ satisfies
\begin{equation}\label{eq_Sylvester_transp}
\Pi^*(-J-R)+L^*B^*=S^*\Pi^*.
\end{equation} Postmultiplying (\ref{eq_Sylvester_transp}) by $Q\Pi$ and adding it to (\ref{eq_Sylvester}) premultiplied by $\Pi^*Q$, yields
\begin{equation}\label{R_G}
\widetilde R=-\frac{1}{2}[(\Pi^*Q\Pi S-GL)+(\Pi^*Q\Pi S-GL)^*],
\end{equation} with $\widetilde R$ as in (\ref{model_pH_Pi_paramters}). Postmultiplying (\ref{eq_Sylvester_transp}) by $Q\Pi$ and subtracting it from (\ref{eq_Sylvester}) premultiplied by $\Pi^*Q$, yields
\begin{equation}\label{J_G}
\widetilde J=\frac{1}{2}[(\Pi^*Q\Pi S-GL)-(\Pi^*Q\Pi S-GL)^*],
\end{equation} with $\widetilde J$ as in (\ref{model_pH_Pi_paramters}). Hence $(\widetilde J-\widetilde R)\widetilde Q(\Pi^*Q\Pi)=(\Pi^*Q\Pi)(S-GL),\ \widetilde B=(\Pi^*Q\Pi) G,$ with $\Pi^*Q\Pi$ invertible, since $Q$ is assumed invertible, proving the claim.
\\
The sufficiency is a direct consequence of Lemma \ref{lema_pH_from_passive_H}, with $P=(\Pi^*Q\Pi)^{-1}$ and $C\Pi=B^*Q\Pi$.

{\it Proof of Statement (2)}. First we show the necessity.
Let $\Upsilon$ satisfy (\ref{eq_Sylvester_Y}). Since $-J=J^*$ and $R=R^*$, $\Upsilon^*$ satisfies
\begin{equation}\label{eq_Sylvester_Y_transp}
\Upsilon^*{\mathcal Q}^*=Q(-J-R)\Upsilon^*+QB{\mathcal R}^*.
\end{equation} Postmultiplying (\ref{eq_Sylvester_Y_transp}) by $Q^{-1}\Upsilon^*$ and adding it to (\ref{eq_Sylvester_Y}) premultiplied by $\Upsilon Q^{-1}$, yields
\begin{equation*}
\widetilde R=-\frac{1}{2}[({\mathcal Q}\Upsilon Q^{-1}\Upsilon^*-{\mathcal R}B^*\Upsilon^*)+({\mathcal Q}\Upsilon Q^{-1}\Upsilon^*-{\mathcal R}B^*\Upsilon^*)^*],
\end{equation*} with $\widetilde R$ as in (\ref{model_pH_Y_paramters}). Postmultiplying (\ref{eq_Sylvester_Y_transp}) by $Q^{-1}\Upsilon^*$ and subtracting it from (\ref{eq_Sylvester_Y}) premultiplied by $\Upsilon Q^{-1}$, yields
\begin{equation*}
\widetilde J=\frac{1}{2}[({\mathcal Q}\Upsilon Q^{-1}\Upsilon^*-{\mathcal R}B^*\Upsilon^*)-({\mathcal Q}\Upsilon Q^{-1}\Upsilon^*-{\mathcal R}B^*\Upsilon^*)^*],
\end{equation*} with $\widetilde J$ as in (\ref{model_pH_Y_paramters}). Hence
\begin{equation*}
(\widetilde J-\widetilde R)\widetilde Q={\mathcal Q}-{\mathcal R}H,\ \widetilde B=\Upsilon B,
\end{equation*} with $\widetilde Q=(\Upsilon Q^{-1}\Upsilon^*)^{-1}$ invertible, since $Q$ is assumed invertible, proving the claim.
\\
The sufficiency is a direct consequence of Lemma \ref{lema_pH_from_passive_H}, with $P=\Upsilon Q^{-1}\Upsilon^*$ and $P^{-1}B\Upsilon=H^*$.
\end{proof}
If $S$ and $L$ are fixed, then $\Sigma_\Pi$ is the unique model from the class $\Sigma_G$ that matches the moments of (\ref{pH_system}) at $\sigma(S)$ and preserves the port Hamiltonian structure. Assume now that $S$ is fixed and let $L=[l_1\ l_2\ \dots\ l_\nu]$, $l_i\in\mathbb{C}$, $i=1,...,\nu$, be such that the pair $(L,S)$ is observable. Then the solution of the Sylvester equation (\ref{eq_Sylvester}) is given by a matrix $\Pi(L)$, yielding a class of reduced order port Hamiltonian models $\Sigma_{\Pi(L)}$ defined by (\ref{model_pH_Pi}) with $\widetilde J(L),\ \widetilde R(L), \widetilde Q(L), \widetilde B(L)$ as in (\ref{model_pH_Pi_paramters}). Note that the input output behaviour is not affected by the choice of $l_1,\ ...,\ l_\nu$, i.e., all models parameterized in $L$ have the same transfer function. However, since the port Hamiltonian structure is a state-space property, the parameters $l_i$, $i=1,...,\nu$ can be used to enforce additional structure, e.g. the Hamiltonian defined by $\widetilde Q$, or the dissipation $\widetilde R$, have a desired form.  Similarly, let ${\mathcal R}=[r_1\ r_2\ \dots\ r_\nu]^*$, $r_i\in\mathbb{C}$, $i=1,...,\nu$, be such that $({\mathcal Q},{\mathcal R})$ is controllable. Then the solution of the Sylvester equation (\ref{eq_Sylvester_Y}) is given by a matrix $\Upsilon({\mathcal R})$, yielding a family of reduced order port Hamiltonian models $\Sigma_{\Upsilon({\mathcal R})}$ defined by equation (\ref{model_pH_Pi}) with $\widetilde J({\mathcal R}), \widetilde R({\mathcal R}), \widetilde Q({\mathcal R}), \widetilde B({\mathcal R})$ as in (\ref{model_pH_Y_paramters}). All models parameterized in ${\mathcal R}$ have the same transfer function and the parameters $r_i$, $i=1,...,\nu$ can be selected such that the models meet additional constraints.

In the MIMO case, $l_i\in\mathbb{C}^m$, $i=1,...,\nu$ can be chosen such that the right tangential interpolation conditions (\ref{eq_tangent}) are satisfied. Similarly, $r_i\in\mathbb{C}^{1\times p}$, $i=1,...,\nu$ can be chosen such that the left tangential interpolation conditions (\ref{eq_tangent_dual}) are satisfied.
\begin{cor}\label{cor_PH_red_tang}
Consider a MIMO, port Hamiltonian system (\ref{pH_system}), with the input $u(t)\in\mathbb{R}^m$ and the output $y(t)\in\mathbb{R}^m$, i.e. $B\in\mathbb{C}^{n\times m}$ and $C\in\mathbb{C}^{m\times n}$. Let $L=[l_1\ ...\ l_\nu]\in\mathbb{C}^{m\times \nu}$ and ${\mathcal R}=[r_1^*\ ...\ r_\nu^*]^*\in\mathbb{C}^{\nu\times m}$. Then the following statements hold.
\begin{enumerate}
\item The MIMO system $\Sigma_\Pi$ described by equations (\ref{model_pH_Pi}) and (\ref{model_pH_Pi_paramters}), with input $u(t)$ and output $\psi(t)\in\mathbb{R}^m$, satisfies the right interpolation conditions (\ref{eq_tangent}).

\item The system $\Sigma_\Upsilon$ described by equations (\ref{model_pH_Pi}) and (\ref{model_pH_Y_paramters}), with input $u(t)$ and output $\psi(t)\in\mathbb{R}^m$, satisfies the left interpolation conditions (\ref{eq_tangent_dual}).\fin
\end{enumerate}
\end{cor}
Theorem \ref{prop_pH_inc_G} offers a way to find a reduced order port Hamiltonian model, from a reduced order model that achieves matching of moments of the given (port Hamiltonian) system, by selecting the parameter $G$. Let (\ref{model_gen_G}) be a reduced order model and let $P$ be such that $S^*P+PS\leq \Pi^*QBL+L^*B^*Q\Pi$. Then, according to Theorem \ref{thm4_[4]}, there exists $G$ such that the model is passive, i.e., $PG=\Pi^*QB$. Selecting $P=\Pi^*Q\Pi$ yields the parameter $G$ which identifies the port Hamiltonian reduced order model that achieves moment matching. Based on Lemma \ref{lema_pH_from_passive_H}, similar arguments hold for the case of the system $\Sigma_\Upsilon$, with $P=\Upsilon Q^{-1}\Upsilon^*$.

Lemma \ref{lemma_V_and_Pi} implies that the result from Theorem \ref{prop_pH_inc_G} is equivalent to the result from Theorem \ref{thm_red_PH_WV}, i.e., the class of reduced order models obtained by Krylov projection and the class of reduced order models obtained by time-domain moment matching are equivalent. In detail, let $\Sigma_G$ and $\Sigma_V$ be two reduced order models of (\ref{pH_system}). Then selecting $T=\Pi^*Q\Pi$ yields $\Pi^*Q\Pi(S-GL)=\widetilde W^*(J-R)Q\widetilde V \Pi^*Q\Pi$, $\Pi^*Q\Pi G=\widetilde W^* B$ and $B^*Q\Pi=B^*Q\widetilde V\Pi^*Q\Pi$, which shows that one port Hamiltonian model can be obtained from the other via a coordinate transformation. Note that from an input-output point of view, both models exhibit the same behaviour. Similar arguments hold for the case of the system $\Sigma_\Upsilon$, with $T=\Upsilon Q^{-1}\Upsilon^*$. Hence, from a computational point of view, Theorem \ref{prop_pH_inc_G} offers a way to compute a reduced order model $\Sigma_\Pi$ (or $\Sigma_\Upsilon$) that achieves moment matching and preserves the port Hamiltonian structure.

\begin{alg}\label{alg_PHred_finite}(Computation of a port Hamiltonian reduced order model that matches a prescribed number of moments of a given port Hamiltonian linear system, at a set of finite interpolation points.) \\
\n
1. Select $S$ and $L$ such that the pair $(L,S)$ is observable. \\
2. Use any efficient algorithm (e.g. Iterative Rational Krylov) to compute $V$ depending on the eigenvalues of $S$. \\
3. According to Lemma \ref{lemma_V_and_Pi}, set $\Pi=V$. \\
4. Compute the class of reduced order models $\Sigma_G$ as in (\ref{model_gen_G}). \\
5. Let $G=(\Pi^*Q\Pi)^{-1}\Pi^*QB$ and compute the port Hamiltonian model $\Sigma_\Pi$ as in (\ref{model_pH_Pi}) and (\ref{model_pH_Pi_paramters}).\fin
\end{alg}
The outcome of Algorithm \ref{alg_PHred_finite} is a port Hamiltonian model that matches the moments of (\ref{pH_system}) at $\sigma(S)$, parameterized in $S$ and $L$. Hence, following arguments from Section \ref{sect_krylov}, a suitable selection of the interpolation points, at step 1, such that the model approximates (\ref{pH_system}) in the $H_2$ norm more accurately, is given by the matrix $S$ such that $\sigma(S)=\{-\lambda_1,...,-\lambda_\nu\}$, where $\lambda_i\in\sigma((J-R)Q)$.

\subsection{Markov parameter matching with preservation of port Hamiltonian structure}\label{sect_pH_MarkovM}

 In this section we solve another particular instance of \ref{prob_redmod_prose}, for the case of matching a set of prescribed Markov parameters.

\begin{prob}\label{prob_redmod_Markov}\n
Consider the linear, port Hamiltonian, SISO system (\ref{pH_system}) and the observable pair $(L,S)$, where $L\in\mathbb{C}^{1\times \nu}$ and $S\in\mathbb{C}^{\nu\times\nu}$ such that Assumption \ref{ass_eigenvalues_Markov} holds. Find $F\in\mathbb{C}^{\nu\times\nu}$ and/or $G\in\mathbb{C}^\nu$, such that the following properties hold.
\begin{description}
\item[(m1)] The family of systems $\Sigma_{\mathbf\Pi}$ parameterized in $F$ and/or $G$ and $H$, described by equations (\ref{red_mod_FGH}), with $\mathbf\Pi\in\{\Pi,\bar\Pi\}$, where $\Pi$ is the unique solution of (\ref{eq_matrix_Pi_Markov}) and $\bar\Pi$ is the unique solution of (\ref{eq_matrix_Pi_Markov_bar}) match the moments of the transfer function $\widetilde K(\tau)$ of (\ref{pH_system}) at $\sigma(S)$, i.e., conditions (\ref{cond_FGH}) or (\ref{cond_FGH_1}) are satisfied.

\item[(m2)] There exists $G$ and $H$ such that the family of systems $\Sigma_{\mathbf\Pi}$, described by equations (\ref{red_mod_FGH}), preserve the port Hamiltonian structure.


\item[(m3)] The computation of the family of systems $\Sigma_{\mathbf\Pi}$ described by equations (\ref{red_mod_FGH}) avoids the explicit solution of the Markov parameters. \fin
\end{description}
\end{prob}

Based on the solution of subproblem {\bf (m1)} provided in this section, we provide solutions to subproblem {\bf (m2)}, suitable to the scope of this paper. We give some insight into subproblem {\bf (m3)} provided by Proposition \ref{prop_equiv_Markov} which allows for efficient computation of the families of reduced order models. Note that a similar problem can be formulated in terms of the results of Section \ref{sect_Markov_dual}, hence it is omitted.

In this section we focus on Problem \ref{prob_redmod_Markov}, i.e., find a reduced order model that matches a prescribed number of Markov parameters and preserves the port Hamiltonian structure. Consider the port Hamiltonian system (\ref{pH_system}) with the transfer function $K(s)=B^*Q(sI-(J-R)Q)^{-1}B$. Suppose Assumption \ref{ass_eigenvalues_Markov} holds.  The moments of $\widetilde K(\tau)=K(1/\tau)$ at $\tau=\tau^*$ are in a one-to-one relation with $B^*Q\Pi S$, where $\Pi$ is the unique solution of the Sylvester equation
\begin{equation}\label{eq_Sylv_tau_pH}
(J-R)Q\Pi S+BL=\Pi.
\end{equation} In addition, let $\bar\Pi$ be the unique solution of the Sylvester equation
\begin{equation}\label{eq_Sylv_tau_pH_bar}
(J-R)Q\bar\Pi S+BLS=\bar\Pi.
\end{equation} The moments of $\widetilde K(\tau)=K(1/\tau)$ at $\tau=\tau^*$ are in a one-to-one relation with $B^*Q\bar\Pi$. The first $\nu$ Markov parameters of (\ref{pH_system}) are the moments of $\widetilde K(\tau)$ for $\tau^*=0$. Assume there exists an invertible matrix $P$ such that a reduced order model described by equations (\ref{red_mod_FGH}) exists and the relations (\ref{cond_FGH}) are satisfied. Furthermore, assume there exists an invertible matrix $\bar P$ such that conditions (\ref{cond_FGH_1}) are satisfied.

At the same time, we compute the class of port Hamiltonian reduced order models that achieve moment matching based on Proposition \ref{prop_mom_Markov_Y}. The moments of $\widetilde K(\tau)=K(1/\tau)$ at $\tau=\tau^*$ are in a one-to-one relation with ${\mathcal Q}\Upsilon B$, where $\Upsilon$ is the solution of the Sylvester equation
\begin{equation}\label{eq_Sylv_tau_Y_pH}
{\mathcal Q}\Upsilon (J-R)Q+{\mathcal R}B^*Q=\Upsilon.
\end{equation} The first $\nu$ Markov parameters of (\ref{pH_system}) are the moments of $\widetilde K(\tau)$ for $\tau^*=0$.
\begin{prop}\label{prop_red_pH_Markov_Pi}
Let $K(s)$ be the transfer function of the system (\ref{pH_system}). Let $(L,S)$ be an observable pair and let $(\mathcal{Q},\mathcal{R})$ be a controllable pair. Consider the port Hamiltonian system
\begin{equation}\label{model_pH_Markov_Pi} \begin{split}
&\dot \xi= (\widetilde J-\widetilde R)\widetilde Q\xi+\widetilde Bu,\\& \psi=\widetilde B^*\widetilde Q\xi,
\end{split}\end{equation} with $\xi(t)\in\mathbb{R}^\nu$. The following statements hold.
\begin{enumerate}
\item Let
\begin{equation}\label{model_pH_Markov_Pi_paramters}\begin{split}
&\widetilde J=\mathbf{\Pi}^*QJQ\mathbf{\Pi},\ \widetilde R=\mathbf{\Pi}^*QRQ\mathbf{\Pi},\\& \widetilde Q=(\mathbf{\Pi}^*Q\mathbf{\Pi})^{-1},\ \widetilde B=\mathbf{\Pi}^*QB,
\end{split}\end{equation} with $\mathbf{\Pi}=\Pi$, or $\mathbf{\Pi}=\bar\Pi$, where $\Pi$ is the unique solution of equation (\ref{eq_Sylv_tau_pH}) and $\bar\Pi$ is the unique solution of equation (\ref{eq_Sylv_tau_pH_bar}). Let $\Sigma_{\mathbf \Pi}$ denote a port Hamiltonian model described by equations (\ref{model_pH_Markov_Pi}) and (\ref{model_pH_Markov_Pi_paramters}). If $\lambda\mu\ne 1$, for any $\lambda\in\sigma(S)$ and any $\mu\in\sigma((\widetilde J-\widetilde R)\widetilde Q)$, then the
system $\Sigma_{\mathbf \Pi}$ is a port Hamiltonian, reduced order model that matches the moments of $\widetilde K(\tau)=K(1/s)$ at $\sigma(S)$.

\item  Let
\begin{equation}\label{model_pH_Markov_Y_parameters}\begin{split}
&\widetilde J=\Upsilon^*J\Upsilon,\ \widetilde R=\Upsilon^*R\Upsilon,\\& \widetilde Q=(\Upsilon^*Q^{-1}\Upsilon)^{-1},\ \widetilde B=\Upsilon^*B,
\end{split}\end{equation} with $\Upsilon$ the unique solution of equation (\ref{eq_Sylv_tau_Y_pH}). Let $\Sigma_{\Upsilon}$ denote a port Hamiltonian model described by equations (\ref{model_pH_Markov_Pi}) and (\ref{model_pH_Markov_Y_parameters}). If $\lambda\mu\ne 1$, for any $\lambda\in\sigma(\mathcal Q)$ and any $\mu\in\sigma((\widetilde J-\widetilde R)\widetilde Q)$, then the system $\Sigma_\Upsilon$ is a port Hamiltonian, reduced order model that matches the moments of $\widetilde K(\tau)=K(1/s)$ at $\sigma({\mathcal Q})$. \fin
\end{enumerate}
\end{prop}
\begin{proof}
{\it Proof of Statement (1)}. Let $\mathbf{\Pi}=\Pi$ be the unique solution of equation (\ref{eq_matrix_Pi_Markov}). Let $P=\widetilde Q^{-1}=\Pi^*Q\Pi$, then $(\widetilde J-\widetilde R)\widetilde Q PS+\widetilde BL-P=\Pi^*Q(J-R)Q\Pi(\Pi^*Q\Pi)^{-1}\Pi^*Q\Pi S+\Pi^*QBL-\Pi^*Q\Pi=\Pi^*Q[(J-R)Q\Pi S+BL-\Pi]$. By (\ref{eq_matrix_Pi_Markov}) we have that $(J-R)Q\Pi S+BL-\Pi=0$, hence $(\widetilde J-\widetilde R)\widetilde Q PS+\widetilde BL=P$. Furthermore $\widetilde B^*\widetilde Q P=B^*Q\Pi(\Pi^*Q\Pi)^{-1}\Pi^*Q\Pi=B^*Q\Pi$. Then, system (\ref{model_pH_Markov_Pi}) satisfies conditions (\ref{cond_FGH}) and the claim is proven. The proof is similar in the case $\mathbf{\Pi}=\bar\Pi$.
\\
The proof of the second statement follows similar arguments as the proof of Proposition \ref{prop_red_pH_Pi}, hence it is omitted.
\end{proof}
Proposition \ref{prop_red_pH_Markov_Pi} shows that there exists a port Hamiltonian model that matches the moments of a given port Hamiltonian system by direct computation of the matrix $\mathbf \Pi$. This computation is avoided by showing that the model $\Sigma_{\mathbf\Pi}$ is a member of the class of reduced order models (\ref{red_mod_FGH}) that match the moments of $\widetilde K(\tau)=K(1/\tau)$ at $\sigma(S)$, for a particular instance of the parameter $G$ and of the output $H$. Similarly, we also show that the model $\Sigma_{\mathbf\Upsilon}$ is a member of the class of reduced order models (\ref{red_FGH_left}) that match the moments of $\widetilde K(\tau)=K(1/\tau)$ at $\sigma(\mathcal{Q})$, for a particular selection of $H$.
\begin{thm}\label{prop_Markov_pH_inc_G} The following statements hold.
\begin{enumerate}
\item Let system $\Sigma_{\mathbf \Pi}$, as in (\ref{red_mod_FGH}), be a reduced order model of system (\ref{pH_system}) that matches the moments of $\widetilde K(\tau)=K(1/s)$ at $\sigma(S)$, where $K(s)$ is the transfer function of (\ref{pH_system}). Then $\Sigma_{\mathbf\Pi}$ is equivalent to a port Hamiltonian system $\Sigma_{\mathbf\Pi}$, described by equations (\ref{model_pH_Markov_Pi}) and (\ref{model_pH_Markov_Pi_paramters}), if and only if $G={\mathbf\Pi}^*QB$ and $H=G^*({\mathbf \Pi}^*Q{\mathbf \Pi})^{-1}$.

\item Let system (\ref{red_FGH_left}) be a model that matches the moments of $\widetilde K(\tau)=K(1/s)$ at $\sigma(\mathcal Q)$, where $K(s)$ is the transfer function of (\ref{pH_system}). If $E=I$, then (\ref{red_FGH_left}) is equivalent to a port Hamiltonian system $\Sigma_\Upsilon$ described by equations (\ref{model_pH_Markov_Pi}) and (\ref{model_pH_Markov_Y_parameters}), if and only if $H=B^*\Upsilon^*\widetilde Q$.\fin
\end{enumerate}
\end{thm}
\begin{proof} {\it Proof of Statement (1)}. The necessity follows from the fact that (\ref{model_pH_Markov_Pi}) matches the moments of (\ref{pH_system}). Let $\mathbf\Pi=\Pi$ satisfy (\ref{eq_Sylv_tau_pH}). Postmultiplying the transpose of (\ref{eq_Sylv_tau_pH}) by $Q\Pi$ and adding/subtracting it to/from (\ref{eq_Sylv_tau_pH}) premultiplied by $\Pi^*Q$, yields
\begin{subequations}
\begin{align}
(\widetilde J-\widetilde R) S+S^*(\widetilde J+\widetilde R) &= L^*\widetilde B-\widetilde B L, \label{eq1}\\
(\widetilde J-\widetilde R) S-S^*(\widetilde J+\widetilde R) &= 2\Pi^*Q\Pi-L^*\widetilde B-\widetilde B L, \label{eq2}
\end{align}\end{subequations} with $\widetilde R$ and $\widetilde J$ as in (\ref{model_pH_Markov_Pi_paramters}). Furthermore, adding equations (\ref{eq1}) and (\ref{eq2}) yields
\begin{equation*}
(\widetilde J-\widetilde R)\widetilde Q \Pi^*Q\Pi S=\Pi^*Q\Pi-\widetilde BL,
\end{equation*} with $\Pi^*Q\Pi$ invertible, since $Q$ is assumed invertible. Moreover, $B^*Q\Pi(\Pi^*Q\Pi)^{-1}\Pi^*Q\Pi S=B^*Q\Pi S$. Hence (\ref{model_pH_Markov_Pi}) is a model described by equations (\ref{red_mod_FGH}) with $F=(\widetilde J-\widetilde R)\widetilde Q,\ G=\widetilde B,\ H=G^*P^{-1}$ and $P$ satisfying $P=\widetilde Q^{-1}=\Pi^*Q\Pi$ from the matching conditions (\ref{cond_FGH}). \\
The sufficiency is proven in two parts. First we assume that $S$ has full rank. Let (\ref{red_mod_FGH}) be a reduced order model of (\ref{pH_system}), with $G=\Pi^*QB$ and $H=G^*(\Pi^*Q\Pi)^{-1}$. System (\ref{red_mod_FGH}) matches the moments of $\widetilde K(\tau)$, since there exists $P=\Pi^*Q\Pi$ such that $F\Pi^*Q\Pi S+\Pi^*QBL=\Pi^*Q\Pi,\ H\Pi^*Q\Pi S=B^*Q\Pi S$. In addition, we have $F\Pi^*Q\Pi S=\Pi^*Q(\Pi-BL)$, which by (\ref{eq_Sylv_tau_pH}) becomes $[F-(\widetilde J-\widetilde R)\widetilde Q]\Pi^*Q\Pi S=0$. Since $S$ is assumed to have full rank, $F=(\widetilde J-\widetilde R)\widetilde Q$, i.e. (\ref{red_mod_FGH}) is a port Hamiltonian system described by equations (\ref{model_pH_Markov_Pi}). If $S$ does not have full rank, assume, without loss of generality that $S=\left[\begin{array}{cc} 0 & S_1 \\ 0 & S_2 \end{array}\right]$, with $S_2$ square, of appropriate dimensions and full rank, and $L=[1\ 0\ \dots\ 0]$. Let $\Pi=[0\ \widetilde\Pi]$, with $\widetilde\Pi$ (having appropriate dimensions) satisfying the Sylvester equation $(J-R)Q\widetilde\Pi S_2+BS_1=\widetilde\Pi$. Assume (\ref{red_mod_FGH}) is a reduced order model with $G=\widetilde\Pi^* QB$ and $H=G^*(\widetilde \Pi^*Q\widetilde\Pi)^{-1}$. Since $S_2$ has full rank, the arguments from the previous case are followed. Hence, applying the matching conditions (\ref{cond_FGH_tilde}) yields $[F-(\widetilde J-\widetilde R)\widetilde Q]\widetilde\Pi^*Q\widetilde\Pi S_2=0$, which leads to the claim. The proof is similar for the case ${\mathbf \Pi}=\bar\Pi$.
\\
The proof of the second statement follows similar arguments as the proof of Theorem \ref{prop_pH_inc_G}, for $P=I=\widetilde Q^{-1}\widetilde Q$, where $P$ is an invertible matrix uniquely satisfying the relation (\ref{markov_match_cond_1}).
\end{proof}
If $\tau^*=0$, the models $\Sigma_{\mathbf \Pi}^{\rm pH}$ match the first $\nu$ Markov parameters of (\ref{pH_system}) and preserve the port Hamiltonian structure of the given system. Furthermore, a direct application of Proposition \ref{prop_equiv_Markov} yields that $\Sigma_V=\Sigma_{\Pi}$, with $\Sigma_V$ as in (\ref{model_pH_WV}) and (\ref{model_pH_WV_paramters}).

Based on Theorem \ref{prop_Markov_pH_inc_G} and on Proposition \ref{prop_equiv_Markov} we propose the following procedure.


\begin{alg}\label{alg_PHred_Markov}(Computation of a port Hamiltonian reduced order model that matches a prescribed number of Markov parameters of a given port Hamiltonian linear system.) \\
\n
1. Select a non-derogatory matrix $S$, such that $\sigma(S)=\underbrace{\{0,...,0\}}_{\nu}$, and $L$ such that the pair $(L,S)$ is observable. \\
2. Use any efficient algorithm to compute $V$ as in (\ref{eq_V_Markov}). \\
3. Set $\mathbf\Pi=V$. \\
4. Compute the class of reduced order models $\Sigma_{\mathbf\Pi}$ as in (\ref{red_mod_FGH}). \\
5. Let $G={\mathbf\Pi}^*QB$ and $H=G^*({\mathbf \Pi}^*Q{\mathbf \Pi})^{-1}$ and compute the port Hamiltonian model $\Sigma_{\mathbf\Pi}$ as in (\ref{model_pH_Markov_Pi}) and (\ref{model_pH_Markov_Pi_paramters}).\fin
\end{alg}

The outcome of Algorithm \ref{alg_PHred_Markov} is a port Hamiltonian model that matches the Markov parameters of (\ref{pH_system}), parameterized in $L$. The parameter $L$ cam be used to enforce further additional structure on the state-space realization of the reduced order model, increasing the physical meaning of the approximant. Note that if the result of Algorithm \ref{alg_PHred_Markov} is $\Sigma_{\bar\Pi}$ with $\bar\Pi$ the unique solution of (\ref{eq_matrix_Pi_Markov_bar}), then we obtain a $\nu$-th order model that matches $\nu+1$ Markov parameters, see also Example \ref{example_markov_match}.

Note that the results of both Theorem \ref{prop_pH_inc_G} and Theorem \ref{prop_Markov_pH_inc_G} apply to the port Hamiltonian system in any coordinate system, resulting in equivalent state-space port Hamiltonian reduced order models, without the need to compute additional transformations.

\section{Illustrative examples}\label{sect_example}

\subsection{SISO ladder network}

\begin{figure}[h]\centering
\psfrag{I}{$u=I$} \psfrag{R1}{$R_1$} \psfrag{R2}{$R_2$}
\psfrag{L1}{$L_1,\phi_1$} \psfrag{L2}{$L_2,\phi_2$} \psfrag{C1}{$C_1,q_1$}
\psfrag{C2}{$C_2,q_2$} \psfrag{R3}{$R_3$} \psfrag{Y}{$y=V_{C_1}$}
\includegraphics[scale=1.1]{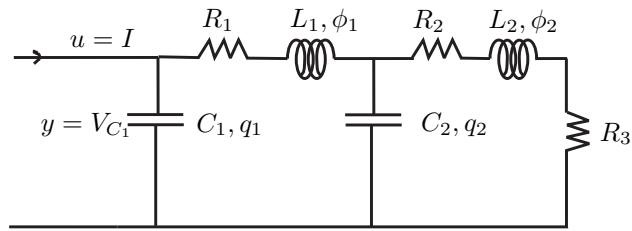} \caption{Fourth order ladder network.} \label{example_ladder}
\end{figure}
Consider the ladder network in Fig. \ref{example_ladder}, with $C_1,\ C_2,\ L_1,\ L_2,\ R_1,\ R_2$ the capacitances, inductances, and resistances of the corresponding capacitors, inductors, and resistors, respectively. The port Hamiltonian representation of this system is given by equations of the form (\ref{pH_system}), with $x=[q_1\ \phi_1\  q_2\ \phi_2]^*$ and
\begin{equation}\label{exmp_ladder} \begin{split}
&J=\left[\begin{array}{cccc} 0& -1 & 0 & 0 \\ 1 & 0 & -1 & 0 \\ 0 & 1 & 0 & -1 \\ 0 & 0 & 1 & 0 \end{array}\right],\ R=\diag\{0,R_1,0,R_2+R_3\}, \ Q=\diag\left\{\frac{1}{C_1},\frac{1}{L_1},\frac{1}{C_2},\frac{1}{L_2}\right\},\ B=[1\ 0\ 0\ 0]^*.
\end{split}\end{equation}
Note that $\phi$ denotes the flux through the inductor $L$ and $q$ denotes the charge at the capacitor $C$.
Assume $C_1=C_2\ne 0$, $L_1=L_2\ne 0$, $R_1=R_2=R_3\ne 0$. The transfer function of the port Hamiltonian system (\ref{exmp_ladder}) is
\begin{equation*}\begin{split}
 K(s) =\frac{L_1 ^{2}C_1s ^{3}+3L_1 R_1 C_1 s ^{2}+\left(2L_1 +2R_1 ^{2}C_1 \right)s +3R_1 }{C_1 ^{2}L_1 ^{2}s ^{4}+3C_1 ^{2}L_1 R_1s ^{3}+\left(3C_1 L_1 +2C_1 ^{2}R_1 ^{2}\right)s ^{2}+5R_1 C_1 s +1}.
\end{split}\end{equation*}

\paragraph*{Matching at finite interpolation points.} The first two moments of (\ref{exmp_ladder}) at $0$ are $\eta_0=3R_1$ and $\eta_1=2L_1-13R_1^2C_1$. Let $L=[1\ 0]$ and $S=\left[\begin{array}{cc} 0 & 1 \\ 0 & 0 \end{array}\right]$. Note that
$$\Pi=\left[\begin{array}{cc} 3R_1C_1 & C_1(L_1-13C_1R_1^2) \\ L_1 & -3R_1L_1C_1 \\ 2R_1C_1 & C_1(L_1-10R_1^2C_1) \\ L_1 & -5R_1L_1C_1 \end{array} \right].$$
\noindent A reduced order port Hamiltonian model that matches the moments $\eta_0$ and $\eta_1$ is given by (\ref{model_pH_Pi}), with
\begin{equation}\label{exmp_ladder_parameters}\begin{split}
&\widetilde J=\left[\begin{array}{cc} 0& 2L_1\\ -2L_1 & 0\end{array}\right] ,\ \widetilde R=R_1\left[\begin{array}{cc} 3 & -13R_1C_1 \\ -13R_1C_1 & 59R_1^2C_1^2 \end{array}\right],\\& \widetilde Q=\frac{1}{10L_1^3-R_1^2C_1(11L_1^2-44R_1^2L_1-16R_1^4)} \left[\begin{array}{cc} 5L_1^2-R_1^2C_1(38L_1-269R_1^2) & 59R_1^3C_1 \\ 59R_1^3C_1 & \frac{13R_1^2C_1+2L_1}{C_1} \end{array}\right],\ \widetilde B=\left[\begin{array}{c}3R_1\\ 2L_1-13R_1^2C_1\end{array}\right].\end{split}\end{equation}
\noindent The transfer function of the reduced order model is $\displaystyle K_{\Pi}(s)=\frac{a}{b}\frac{s+\frac{d}{a}}{s^2+\frac{c}{b}s+\frac{e}{b}}$, with $a,b,c$ given by
\setlength\arraycolsep{2pt}
\begin{align*}
a &= R_1^2C_1(16C_1^2+28R_1^2L_1C_1-7L_1^2)+8L_1^3,
\\
b &= C_1(10L_1^3-R_1^2C_1(11L_1^2-44R_1^2L_1-16R_1^4)),
\\
c &= R_1C_1(40R_1^4C_1^2+15L_1^2+4R_1^2L_1C_1),
\\
d &= 12R_1(L_1^2+2R_1^4C_1^2),
\\
e &= 4(L_1^2+2R_1^4C_1^2).
\end{align*}
Comparing the port Hamiltonian model $\Sigma_\Pi$ in (\ref{exmp_ladder_parameters}) with the model $\Sigma_V$ in (\ref{model_pH_WV}) yields $\Sigma_\Pi=\Sigma_{-V}$, i.e., $\Pi=-V$, where $V$ spans a Krylov subspace of the given port Hamiltonian system given by (\ref{exmp_ladder}) (see also Proposition \ref{prop_red_PH_V}). Note that this yields a reduced order model with the transfer function $K_{\Pi}(s)$.

Let $C_1=1, C_2=2, L_1=L_2=1$ and $R_1=R_2=R_3=1$. Furthermore, let $L=[l_1\ l_2]$, $l_1\in\mathbb{R},\ l_2\in\mathbb{R}$. The pair $(L,S)$ is observable if and only if $l_1\ne 0$. Note that
$$\Pi(l_1,l_2)=\left[\begin{array}{cccc} 3l_1 & l_1 & l_1 & l_1 \\ 3l_2-l_1 & l_2-3l_1 & l_2-\frac{7}{2}l_1 & l_2-4l_1 \end{array}\right]^*.$$
The class of port Hamiltonian models, parameterized in $l_1$ and $l_2$ is given by
\begin{equation*}\label{exmp_ladder_par_L}
\begin{split}
&\widetilde J(l_1,l_2)=\left[\begin{array}{cc} 0 & 2l_1^2 \\ -2l_1^2 & 0 \end{array}\right],\ \widetilde R(l_1,l_2)=\left[\begin{array}{cc} 3l_1^2& 3l_2l_1-11l_1^2 \\ 3l_2l_1-11l_1^2 & 3l_2^2-22l_1l_2+41l_1^2 \end{array}\right],\\& \widetilde Q(l_1,l_2)=\frac{1}{31l_1^4}\left[\begin{array}{cc} 26l_2^2-164l_1l_2+261l_1^2 & 2l_1(41l_1-13l_2) \\ 2l_1(41l_1-13l_2) & 26l_1^2 \end{array} \right],\ \widetilde B(l_1,l_2)=[3l_1\ \ 3l_2-9l_1]^*.\end{split}
\end{equation*}
For $l_2=\frac{41}{13}l_1$, we obtain the subset of reduced order models with the following properties: they match the first two moments of (\ref{exmp_ladder}) at zero, preserve the port Hamiltonian structure of the model and have diagonalized Hamiltonians. For $l_2=\frac{11}{3}l_1$, we obtain a subset of port Hamiltonian reduced order models with diagonal dissipation matrix. All the parameterized models have the same input-output behaviour described by the transfer function $K_{\Pi}(s)=\frac{9(3s+4)}{31s^2+45s+12}$. Note that the physical meanings of the states of the approximant are more difficult to recover using only the parameters $l_i$, $i=1,2$. However, the selection of the interpolation points $s_i$, $i=1,2$ (as free parameters), can bring additional insight into the physics of the approximant.
\paragraph*{"Dual" family of port Hamiltonian reduced order models.} Let ${\mathcal R}=[r_1\ r_2]^*$, $r_1\in\mathbb{R},\ r_2\in\mathbb{R}$. The pair $({\mathcal Q},{\mathcal R})$ is controllable if and only if $r_1\ne 0$. Solving (\ref{eq_Sylvester_Y}), we obtain
\begin{equation*}\Upsilon(r_1,r_2)=\left[\hspace{-0.1cm}\begin{array}{cccc} 3r_1 & -r_1 & 2r_1 & -r_1 \\ 3r_2-r_1 & 3r_1-r_2 & 2r_2-7r_1 & 4r_1-r_2 \end{array}\hspace{-0.1cm}\right].\end{equation*}
\noindent The class of port Hamiltonian models, all with transfer function $K_\Upsilon(s)=\frac{9(3s+2)}{32s^2+27s+6}$, parameterized in $r_1$ and $r_2$ is given by
\begin{equation*}\label{exmp_ladder_par_R}
\begin{split}
&\widetilde J(r_1,r_2)=\left[\begin{array}{cc} 0 & -2r_1^2 \\ 2r_1^2 & 0 \end{array}\right],\ \widetilde R(r_1,r_2)=\left[\begin{array}{cc} 3r_1^2& 3r_2r_1-11r_1^2 \\ 3r_2r_1-11r_1^2 & 3r_2^2-22r_1r_2+41r_1^2 \end{array}\right],\\& \widetilde Q(r_1,r_2)=\frac{1}{32r_1^4}\left[\begin{array}{cc} 204r_1^2-124r_1r_2+19r_2^2 & r_1(62r_1-19r_2) \\ r_1(62r_1-19r_2) & 19r_1^2 \end{array}\right],\ \widetilde B(r_1,r_2)=[3r_1\ \ 3r_2-9r_1]^*.\end{split}
\end{equation*} For $r_2=\frac{62}{19}r_1$ we obtain a subclass of port Hamiltonian reduced order models with diagonal Hamiltonians. For $r_2=\frac{11}{3}r_1$ we obtain a subclass of port Hamiltonian reduced order models with diagonal dissipation matrix.

\paragraph*{Matching at infinity.} Let $L=[l_1\ l_2\ l_3]^*$, $l_1\in\mathbb{R},\ l_2\in\mathbb{R},\ l_3\in\mathbb{R}$. Note that
$$\Pi(l_1,l_2,l_3)=\left[\begin{array}{ccc} l_1 & l_2 & l_3-l_1 \\ 0 & l_1 & l_2-l_1 \\ 0 & 0 & l_1 \\ 0 & 0 & 0\end{array}\right].$$ The family of port Hamiltonian models, parameterized by $l_1,\ l_2,\ l_3,$ is given by equation (\ref{exmp_ladder_Markov_par_L})
\begin{figure*}
\begin{align}\label{exmp_ladder_Markov_par_L}
\widetilde J(l_1,l_2,l_3) &= \left[\hskip -0.1cm \begin{array}{ccc} 0 & -l_1^2 & l_1(l_1-l_2)\\ l_1^2 & 0 & -3l_1^2+l_1l_3-l_2^2+l_1l_3 \\ l_1(l_1-l_2) & -3l_1^2+l_1l_3-l_2^2+l_1l_3 & 0 \end{array} \hskip -0.1cm\right],\ \widetilde R(l_1,l_2,l_3)=\left[\begin{array}{ccc} 0 & 0 & 0 \\ 0 & l_1^2 & l_1(l_1-l_2) \\ 0 & l_1(l_1-l_2) & (l_1-l_2)^2 \end{array}\right], \nonumber\\
\widetilde Q(l_1,l_2,l_3) &= \frac{1}{2l_1^6}\left[\begin{array}{ccc} q_{11}(l_1,l_2,l_3) & q_{12}(l_1,l_2,l_3) & q_{13}(l_1,l_2,l_3)\\ q_{12}(l_1,l_2,l_3) & q_{22}(l_1,l_2,l_3) & q_{23}(l_1,l_2,l_3) \\ q_{13}(l_1,l_2,l_3) & q_{23}(l_1,l_2,l_3) & q_{33}(l_1,l_2,l_3) \end{array}\right],
\ \widetilde B(l_1,l_2,l_3)=[l_1\ l_2\ l_3-l_1]^*. \\
\hline \nonumber\end{align}
\end{figure*}
with
\setlength\arraycolsep{2pt}
\begin{align*}
q_{11}(l_1,l_2,l_3) &= 5l_1^2l_2^2 - 2l_1l_3l_2^2+l_2^4-2l_2^3l_1+3l_1^4-2l_1^3l_3 \nonumber\\ & + l_3^2l_1^2-2l_1^3l_2+2l_2l_3l_1^2, \nonumber
\\
q_{12}(l_1,l_2,l_3) &= l_1^4-4l_1^3l_2+l_1^2l_2l_3-l_2^3l_1+2l_1^2l_2^2-l_1^3l_2, \nonumber
\\
q_{13}(l_1,l_2,l_3) &= l_1^2(l_2^2-l_1l_2+l_1^2-l_1l_3), \nonumber
\\
q_{22}(l_1,l_2,l_3) &= l_1^2(l_2^2-2l_1l_2+3l_1^2), \nonumber
\\
q_{23}(l_1,l_2,l_3) &= l_1^3(l_1-l_1l_2), \nonumber
\\
q_{33}(l_1,l_2,l_3) &= l_1^4.
\end{align*} The input-output behaviour of the family of models described by (\ref{exmp_ladder_Markov_par_L}) is given by the transfer function $K_{\Pi}(s)=\frac{s^2+s+2}{s(s^2+s+3)}$.\\
Let now $L_2=[l_2\ 0]$, $S_1=[1\ 0]$, $S_2=\left[\begin{array}{cc} 0& 1 \\ 0 & 0 \end{array}\right]$, $L=[l_1\ L_2]$ and $S=\left[\begin{array}{cc} 0 & S_1 \\ 0 & S_2 \end{array}\right]$. Consider now equation (\ref{eq_mom_Markov_tilde}), which in this particular case is
\begin{equation}\label{eq_Sylv_tau_pH_tilde}
\bar\Pi_0=0,\ (J-R)Q\widetilde\Pi S_2+Bl_1 S_1+BL_2S_2=\widetilde\Pi.
\end{equation} The second and the third Markov parameters of (\ref{exmp_ladder}) are in one-to-one relation with $B^*Q\widetilde \Pi$, with $\widetilde\Pi (l_1, l_2)=\left[\begin{array}{cccc} l_1 & 0 & 0 & 0 \\ l_2 & l_1 & 0 & 0 \end{array}\right]^*$. A family of reduced order port Hamiltonian models, of dimension two, that match the first three Markov parameters is described by equations (\ref{model_pH_Markov_Pi}), with
\begin{equation*}\begin{split}
& \widetilde J(l_1,l_2)=\left[\begin{array}{cc} 0 & -l_1^2 \\ l_1^2 & 0\end{array}\right],\ \widetilde R(l_1,l_2)=\left[\begin{array}{cc} 0 & 0 \\ 0 & l_1^2 \end{array}\right], \ \widetilde Q(l_1,l_2)=\frac{1}{l_1^4}\left[\begin{array}{cc} l_1^2+l_2^2 & -l_1l_2 \\ -l_1l_2 & l_1^2 \end{array}\right],\ \widetilde B(l_1,l_2)=[l_1\ l_2]^*.
\end{split}\end{equation*} The input-output behaviour of this family of models is given by the transfer function $K_{\Pi}(s)=\frac{s+1}{s^2+s+1}$. Note that selecting $l_1=1$ and $l_2=0$, yields $\tilde Q(l_1)=I$ and $\widetilde B(l_1)=[1\ 0]^*$, rendering the reduced order model a ladder network as in Fig. \ref{example_ladder_redOrd} with the parameters $\widetilde L=\widetilde C=(\widetilde R_1+\widetilde R_2)=1$.
\begin{figure}[h]\centering
\psfrag{I}{$u=I$} \psfrag{R1}{$\widetilde R_1$} \psfrag{R2}{$\widetilde R_2$}
\psfrag{L1}{$\widetilde L,\widetilde\phi$} \psfrag{C1}{$\widetilde C,\widetilde q$}
\psfrag{Y}{$\psi=V_{\widetilde C}$}
\includegraphics[scale=1.1]{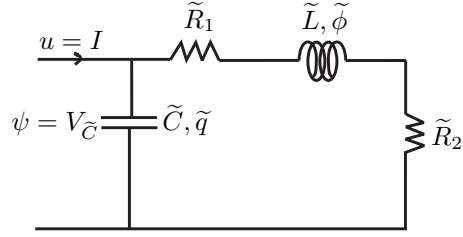} \caption{Second order ladder network.} \label{example_ladder_redOrd}
\end{figure}
Furthermore, let
$${\mathcal Q}=\left[\begin{array}{cc} 0 & 0 \\ 1 & 0 \end{array}\right],\ {\mathcal R}=\left[\begin{array}{c} r_1 \\ r_2 \end{array}\right],$$
with $r_1\ne 0$, $r_2\in\mathbb{R}$. Let $\widehat\Upsilon$ be the unique solution of (\ref{eq_mom_Y_Markov_hat}). The first two Markov parameters of $K(s)$ are in one-to-one relationship with ${\mathcal Q}({\mathcal Q}\widehat\Upsilon + {\mathcal R}B^*Q)B = [0\ r_1]^*$. Consider $\widetilde J$, $\widetilde R$ and $\widetilde Q$ as in (\ref{model_pH_Markov_Y_parameters}), i.e.
\begin{equation*}\begin{split} & \widetilde J=\widehat\Upsilon J \widehat\Upsilon^* = \left[\begin{array}{cc} 0 & 3r_1^2 \\ -3r_1^2 & 0\end{array}\right],\ \widetilde R=\widehat\Upsilon R \widehat\Upsilon^*= \left[\begin{array}{cc} r_1^2 & r_1(r_2-r_1) \\ r_1(r_2-r_1) & (r_1-r_2)^2\end{array}\right], \ \widetilde Q=(\widehat\Upsilon Q^{-1} \widehat\Upsilon^*)^{-1} = \frac{1}{3r_1^2}\left[\begin{array}{cc} \frac{4r_1^2-2r_1r_2+r_2^2}{r_1^2} & \frac{r_1-r_2}{r_1} \\ \frac{r_1-r_2}{r_1} & 1\end{array}\right],\end{split}\end{equation*}
where
$$\widehat\Upsilon=\left[\begin{array}{cccc} 0 & -r_1 & 0 & 0 \\ -r_1 & r_1-r_2 & 2r_1 & 0\end{array}\right]$$
is the unique solution of (\ref{eq_mom_Y_Markov_hat}). Furthermore, solving (\ref{markov_match_cond_2}) yields
$$\widehat P=\diag\{\sqrt{3},\sqrt{3}\},\ \widetilde B=\left[\begin{array}{c} 0 \\ -\sqrt{3}r_1 \end{array}\right].$$
Then a port Hamiltonian reduced order model that matches the first two Markov parameters of the given system $\Sigma_\Upsilon$, described by equations (\ref{model_pH_Markov_Pi}) and (\ref{model_pH_Markov_Y_parameters}) with the transfer function $K_{\Upsilon}(s)=\frac{s+1}{s^2+s+3}$.


\subsection{MIMO single machine infinite bus system}

\begin{figure}[h]
\centering
\includegraphics[scale=0.4]{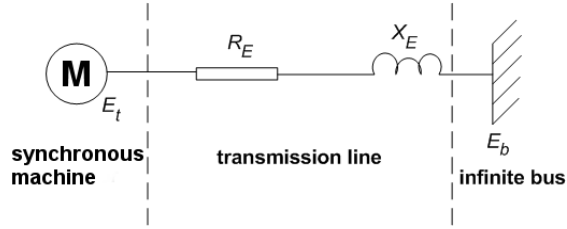}
\caption{A single machine connected to an infinite bus through a
transmission line.}\label{SMIB_fig}
\end{figure}

Consider the model of a single machine connected to an infinite bus (SMIB) useful in the analysis of power systems stability, where a power system is modelled as the interconnection of a large number of such systems. The machine under consideration has one field winding, three stator windings, two $q$-axis amortisseur circuits and one $d$-axis amortisseur circuit, all magnetically coupled consisting of electrical and mechanical equations, see e.g. \cite{kundur-1994, anderson-fouad-1994}.

The nonlinear port Hamiltonian model of the SMIB is given by equations of the form
\begin{equation}\label{apH} \dot x
=(J(x)-R)Qx+Bu, \ y=B^*Qx,
\end{equation}

with
\begin{equation*}\begin{split}
&J(x)=\left[\begin{array}{cccccccc}
0 & \omega X_E & 0 & 0 & 0 & 0 & \Psi_q
\\
-\omega X_E & 0 & 0 & 0 & 0 & 0 & -\Psi_d
\\
0 & 0 & 0 & 0 & 0 & 0 & 0
\\
0 & 0 & 0 & 0 & 0 & 0 & 0
\\
0 & 0 & 0 & 0 & 0 & 0 & 0
\\
0 & 0 & 0 & 0 & 0 & 0 & 0
\\
-\Psi_q & \Psi_d & 0 & 0 & 0 & 0 & 0
\end{array}\right], \ B=\left[\begin{array}{ccc}
0& 0 & \sin\delta \\ 0& 0 & \cos\delta \\ 0 & 1 &0\\ 0 & 0 & 0 \\ 0 & 0 & 0 \\ 0 & 0 & 0 \\ 1 & 0& 0\end{array}\right]\\& R=\diag\{R_a+R_E,R_a+R_E,R_{fd},R_{1d},R_{1q},R_{2q},K_d\},
\end{split}\end{equation*} and $Q=\diag \left\{{\mathbf L}^{-1},\frac{1}{j}\right\}$, where

\begin{itemize}

\item $\delta$ and $\omega$ are the angle and the angular velocity of the rotor, respectively,

\item $R_E$ and $ X_E$ are transmission line resistance and reactance, respectively,

\item $\Psi_d$ and $\Psi_q$ are the stator fluxes in the $d$-axis and $q$-axis, respectively,

\item $R_a$ is the stator resistance, $R_{fd}$ is the field circuit resistance, $R_i, i\in\{1d,1q,2d\}$ are the amortisseurs resistances,

\item $K_d$ is the damping constant and $j$ is the inertia of the rotor,

\item ${\mathbf L}>0$ is the inductance matrix, see, e.g., \cite{ionescu-scherpen-CDC2008} for further details.
\end{itemize}


The state is $x=[\Psi^*\ \omega]^*\in\mathbb{R}^7$, where $$\Psi^*=\left[ \Psi_{ds} \ \Psi_{qs}\ \Psi_{fd} \Psi_{1d} \ \Psi_{1q} \ \Psi_{2q}\right],$$ with

\begin{itemize}

\item  $\Psi_{ds}=\Psi_d+X_Ei_d$ and $\Psi_{qs}=\Psi_q+X_E i_q$, where $i_d$ and $i_q$ are the stator currents,

\item $\Psi_{fd}$ the field flux,

\item $\Psi_{i}, \ i\in\{1d,1q,2d\}$ the rotor fluxes due to the amortisseurs.

\end{itemize}

The input is $u=\left[E_b\ e_{fd}\ T_m\right]^*\in\mathbb{R}^3$, where $E_b$ is the infinite bus voltage (interconnection variable), $e_{fd}$ is the field voltage (control variable) and $T_m$ is mechanical input power. The (passive) output is $y=B^*Qx=[I_b\ i_{fd}\ \omega_m]^*$, where $I_b$ is the infinite bus current, $i_{fd}$ is the field current and $\omega_m$ is mechanical output angular velocity.

Linearising the model (\ref{apH}) around an equilibrium point $x^*=[\Psi^*\ p^*]^*$ yields a system described by equations of the form (\ref{pH_system}), with $J=0$, defined by $(-RQ,B,B^*Q)$, which is minimal and passive. Note that, we solve the tangential interpolation conditions (\ref{eq_tangent}). To this end, we apply Algorithm \ref{alg_PHred_finite} to the study case defined by the equations (\ref{apH}). We have the following parameters for the machine, taken from e.g. \cite{kundur-1994}
\begin{equation*}\begin{split}
& {\mathbf L}= \left[
\begin {array}{ccccccc}  0.22&0& 0.01& 0.01&0&0
\\0& 0.219&0&0& 0.009& 0.009\\
 0.01&0& 1.825& 1.660&0&0\\ 0.01&0& 1.660& 1.8313
&0&0\\0& 0.009&0&0&0& 0.009\\0&
 0.009&0&0& 0.009& 0.134
\end {array} \right],\ B=\left[\begin{array}{ccccccc}
0& 0 & 0 & 0 & 0 & 0 & 1 \\ 0& 0 & 1 & 0 & 0 & 0 & 0\\ 0.7071 & 0.7071 & 0 & 0 & 0 & 0 & 0 \end{array}\right]^T,
\end{split}\end{equation*}
$R=\diag\ \{0.031, 0.031, 0.0006, 0.0284, 0.00619, 0.023638, 10\}$ and $j=6$.
We consider 0 an equilibrium point which is asymptotically stable. Linearising around this equilibrium point we obtain a minimal asymptotically stable linear realization with $\delta$ as a parameter, see, e.g., \cite{ionescu-scherpen-CDC2008}. Let $S=\diag\ \{0.055,0.01,1.667,0.0021\}$, $l_1=[1\ 0\ 0]^*$, $l_2=[0\ 1\ 0]^*$, $l_3=[0\ 0\ 1]^*$, $l_4=[1\ 0\ 1]^*$ and $L=[l_1\ l_2\ l_3\ l_4]$. Note that $(L,S)$ is observable. Using any numerically efficient algorithm compute $\Pi=V=[(s_1+RQ)^{-1}Bl_1\ (s_2+RQ)^{-1}Bl_2\ (s_3+RQ)^{-1}Bl_3\ (s_4+RQ)^{-1}Bl_4]$. The fourth order linear system that matches the moments of $(-RQ,B,B^*Q)$ at $(L,S)$, in the sense of satisfying the right tangential interpolation conditions (\ref{eq_tangent}) is $(S-GL,G,B^*Q\Pi)$, with
\begin{equation*}\begin{split}
& S-GL=\left[\begin{array}{cccc}
   -1.6667  & -0.0048  & -0.0004  & -1.7220 \\
   -0.0000  & -0.0081  & -0.0002  & -0.0002 \\
    0.0000  & -0.4825  & -0.0875  & -1.7545 \\
   -0.0000  &  0.0047  &  0.0004  &  0.0025 \end{array}\right],\
G=\left[\begin{array}{ccc}
    1.7217  &  0.0048  &  0.0004 \\
    0.0000  &  0.0181  &  0.0002 \\
   -0.0000  &  0.4825  &  1.7545 \\
    0.0000  & -0.0047  & -0.0004\end{array}\right].
\end{split}\end{equation*}
Note that $\Pi$ is the unique solution of (\ref{eq_Sylvester}). Using Theorem \ref{thm_redmod_CPi}, the port Hamiltonian state-space representation of $(S-GL,G,B^*Q\Pi)$ is given by a system (\ref{model_pH_Pi}) with
{\small \begin{equation*}\begin{split}
& \widetilde R=P(S-GL)=
\left[\begin{array}{cccc}
    0.0937 &  -0.0000  &  0.0000  &  0.0967 \\
    0.0000 &  71.6093  & -0.2859  &-65.0309 \\
   -0.0000 &  -0.2859  &  0.0525  &  1.7312 \\
    0.0967 & -65.0309  &  1.7312  &135.8342
\end{array}\right],\
\widetilde B=PG=
\left[\begin{array}{ccc}
    0.0968 &  -0.0000  &  0.0000 \\
    0.0000 & 162.9925  &-73.4077 \\
   -0.0000 &  -0.7246  &  1.9466 \\
    0.0999 &-104.9205  &172.7502 \\
\end{array}\right],
\end{split}\end{equation*}}
with $P=(\Pi^*Q\Pi)^{-1}=\widetilde Q^{-1}$, where
$$
\widetilde Q=
\left[\begin{array}{cccc}
   17.7849  & -0.0000  &  0.0111  & -0.0001 \\
   -0.0000  &  0.0001  &  0.0051  &  0.0000 \\
    0.0111  &  0.0051  &  2.0518  & -0.0108 \\
   -0.0001  &  0.0000  & -0.0108  &  0.0001
\end{array}\right].
$$
Note that $(S-GL,G,B^*Q\Pi)$ is a passive system. For an accurate approximation, following the arguments of Remark \ref{obs_optimal_choice}, the interpolation points chosen were approximations of the mirror images of the reduced order model. The initializing choice were the poles of the fourth order balanced truncation, known to have a good approximation error.

\begin{figure}[tbp]\centering
\subfigure[]{
\includegraphics[scale=0.6]{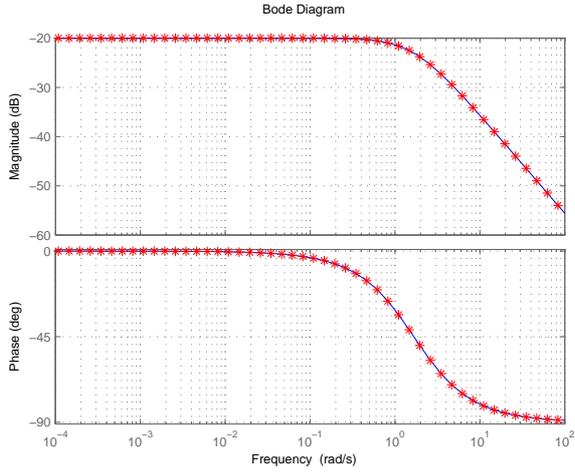} \label{fig_example_smib1}}
\subfigure[]{
\includegraphics[scale=0.6]{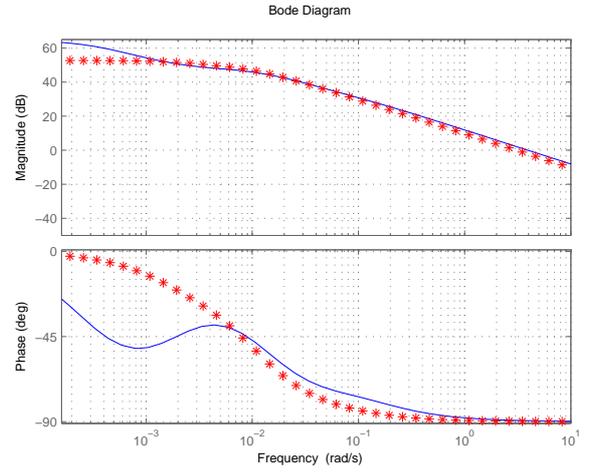}\label{fig_example_smib2}}
\subfigure[]{
\includegraphics[scale=0.6]{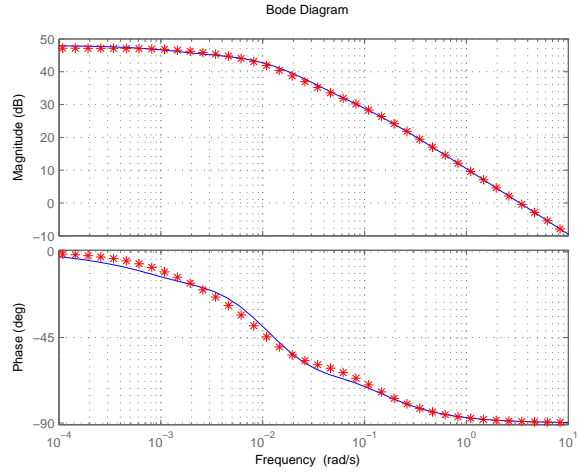}\label{fig_example_smib3}}

\caption{Bode plots of the transfer function between the bus voltage $E_b$ and the bus current $i_b$ (a), from the field voltage $E_{fd}$ to the field current $I_{fd}$ (b) and from the torque $T_m$ to the angular velocity $\omega_m$ (c). The solid line represents the evolutions of the 7th order SMIB and the starred line represents the evolution of the reduced 4th order model}.\label{fig_smib}\end{figure}

Figure \ref{fig_smib} shows the Bode plots of the elements of the transfer matrix that correspond to the transfer between the passive inputs and outputs respectively, i.e., the transfer function from the bus voltage $E_b$ to the bus current $i_b$, as in Fig. \ref{fig_example_smib1}, the transfer function from the field voltage $E_{fd}$ to the field current $I_{fd}$, as in Fig. \ref{fig_example_smib2}, and the transfer function from the torque $T_m$ to the angular velocity $\omega_m$, as in Fig. \ref{fig_example_smib3}, respectively. The plots show both the original and the approximated responses, depicted by solid line and starred line, respectively.

\newpage

\section{Conclusions}

In this paper we have discussed the solutions of the port Hamiltonian structure preserving model reduction Problem \ref{prob_redmod}, based on time-domain moment matching at a set of both finite and infinite points. In the case of matching at finite interpolation points we have first obtained the reduced order port Hamiltonian model that matches the moments of a given port Hamiltonian system (Proposition \ref{prop_red_pH_Pi}). Furthermore, we have characterized \emph{all} the reduced order models which match the moments of the given port Hamiltonian system and preserve the port Hamiltonian structure (Theorem \ref{prop_pH_inc_G}). We have obtained {\it families} of state-space {\it parameterized}, reduced order port Hamiltonian models that approximate the given port Hamiltonian system, all models having the same transfer function. The state-space parameters allow to enforce additional constraints on the structure and/or state-space realization and, in the MIMO case, they have been used to solve the tangential interpolation problem. We have given a possible procedure to compute the family of port Hamiltonian approximations (Algorithm \ref{alg_PHred_finite}).

We have also studied the problem of Markov parameters matching, extending the time-domain moment matching results to the case of interpolation points at infinity. We have defined the moments of a class of linear, descriptor systems, associated to a given linear system, in terms of the unique solutions of Sylvester equations and their dual counterparts (Propositions \ref{prop_mom_Markov} and \ref{prop_mom_Markov_Y}). In particular, the Markov parameters of a given system are the moments of a descriptor realization associated to the given transfer function at zero. Furthermore, we have related the moments to the well-defined steady-state response of the descriptor realization driven by/driving signal generators (Theorems \ref{thm_Markov_time} and \ref{thm_Markov_time_Y}). We have obtained {\it families} of {\it parameterized}, descriptor reduced order models that match a set of prescribed moments of the descriptor realization associated to a given linear system (Proposition \ref{prop_redmod_F_id}). In particular, matching at zero has yielded classes of reduced order models that match the Markov parameters of the given linear system. Finally, applying these results to linear port Hamiltonian systems, we have solved Problem \ref{prob_redmod_Markov}, yielding families of state-space parameterized, reduced order port Hamiltonian models that match the Markov parameters of the given port Hamiltonian system (Proposition \ref{prop_red_pH_Markov_Pi}, Theorem \ref{prop_Markov_pH_inc_G} and Algorithm \ref{alg_PHred_Markov}). Finally, the examples proposed in Section \ref{sect_example} have illustrated the aforementioned results. 

For future work, the nonlinear extension of the results in this paper is a goal. Furthermore, it is very important to determine
how many variables are retained in reduced order model, the relation of the variables between two systems, and the physical meaning of each variable and we will address this issue in the future.


\appendix

\numberwithin{equation}{section} %
\numberwithin{rem}{section} %
\numberwithin{table}{section} %
\numberwithin{prop}{section} %
\numberwithin{lem}{section} %
\numberwithin{thm}{section} %

\section{Preliminaries for the proof of Theorem \ref{prop_Markov_pH_inc_G}}

\begin{thm}\label{thm4_[4]}\cite{astolfi-TAC2010}
The family of $\nu$-th order models $\Sigma_G$ as in (\ref{model_gen_G}) contains a passive system if and only if there exists a symmetric and positive definite matrix $P$ such that
\begin{equation}\label{eq_passive}
S^*P+PS \leq \Pi^*QBL+L^*B^*Q\Pi,
\end{equation} where $\Pi$ is the unique solution of (\ref{eq_Sylvester}).
\fin\end{thm}
\begin{lem}\label{lema_passive}
A family of models (\ref{model_gen_H}) contains a passive model if and only if there exists $P=P^*>0\in\mathbb{R}^{\nu\times\nu}$ such that $P{\mathcal Q}^*+{\mathcal Q}P\leq {\mathcal R}B^*\Upsilon^*-\Upsilon B{\mathcal R}^*$.
\fin\end{lem}
\begin{proof}
The proof follows immediately from applying the Kalman-Yakubovitch-Popov (KYP) lemma to system (\ref{model_gen_H}) (see, e.g., \cite{popov-KYP1962} for an original version and \cite{antoulas-2005,rantzer-SCL1996,willems-1972,rantzer-SCL1996} for extended versions).
\end{proof}
The next result shows how to obtain a port Hamiltonian system from the families of models $\Sigma_G$ and $\Sigma_H$, respectively, described by equations (\ref{model_gen_G}) and (\ref{model_gen_H}), respectively.
\begin{lem}\label{lema_pH_from_passive_H} The following statements hold.
\begin{enumerate}
\item Let $\Sigma_G$, as in (\ref{model_gen_G}), be a passive reduced order model of the system (\ref{pH_system}) and let $P$ satisfy (\ref{eq_passive}). Then the matrices $\widetilde J=\frac{1}{2}[(S-P^{-1}\Pi^*QBL)P^{-1}-P^{-1}(S-P^{-1}\Pi^*QBL)^*]$, $\widetilde R=-\frac{1}{2}[(S-P^{-1}\Pi^*QBL)P^{-1}+P^{-1}(S-P^{-1}\Pi^*QBL)^*]$, $\widetilde Q=P$ and $G=P^{-1}\Pi^*QB$ are such that $\Sigma_G$ is a port Hamiltonian model $\Sigma_\Pi$ described by equations of the form (\ref{model_pH_Pi}) and (\ref{model_pH_Pi_paramters}).

\item Let $\Sigma_H$, as in (\ref{model_gen_H}), be a passive reduced order model of the system (\ref{pH_system}) and let $P$ be as in Lemma \ref{lema_passive}. Then the matrices $\widetilde J=\frac{1}{2}[P({\mathcal Q}-{\mathcal R}H)-({\mathcal Q}-{\mathcal R}H)^*P]$, $\widetilde R=-\frac{1}{2}[P({\mathcal Q}-{\mathcal R}H)+({\mathcal Q}-{\mathcal R}H)^*P]$, $\widetilde Q=P^{-1}$ and $H=(P^{-1}\Upsilon B)^*$ are such that $\Sigma_H$ is a port Hamiltonian model $\Sigma_\Upsilon$ described by equations of the form (\ref{model_pH_Pi}) and (\ref{model_pH_Y_paramters}).\fin
\end{enumerate}
\end{lem}
\begin{proof} The first statement is identical to \cite[Theorem 3]{polyuga-vdschaft-ECC2009}.
The proof of the second statement follows directly from the application of Lemma \ref{lema_passive}.
\end{proof}

\section{Relations between the Krylov projections and the solution of the Sylvester equations}

\begin{lem}\label{lemma_V_and_Pi}\cite{astolfi-CDC2010} The following statements hold.
\begin{enumerate}
\item Consider the matrix $\Pi$, solution of the Sylvester equation (\ref{eq_Sylvester})
and the projector $V$ defined by equation (\ref{eq_V}).
There exists a square, non-singular, matrix $T \in\mathbb{C}^{\nu \times \nu}$ such that $\Pi = V T$.

\item Consider the matrix $\Upsilon$, solution of the Sylvester equation (\ref{eq_Sylvester_Y})
and
the projector $W$ defined as in Theorem \ref{thm_red_PH_WV}.
There exists a square, non-singular, matrix $T \in\mathbb{C}^{\nu \times \nu}$ such that $\Upsilon = T W$. \fin
\end{enumerate}
\end{lem}

\end{document}